\documentclass[11pt,dvipsnames,reqno]{amsart}
\usepackage[left=3.4cm,right=3.4cm,bottom=3.3cm,top=3.3cm]{geometry}
\usepackage{pdflscape}
\usepackage{style}
\usepackage{times}
\usepackage[all]{xy}
\usepackage{multirow}

\usepackage{xcolor,colortbl}

\newcommand{\calO}{\mathcal{O}}

\newcommand{\calH}{\mathcal{H}}

\newcommand{\bbP}{\mathbb{P}}
\newcommand{\bbZ}{\mathbb{Z}}
\newcommand{\bbQ}{\mathbb{Q}}
\newcommand{\bbG}{\mathbb{G}}
\newcommand{\bbF}{\mathbb{F}}

\newcommand{\sfE}{\mathsf{E}}
\newcommand{\sfI}{\mathsf{I}}

\newcommand{\frakS}{\mathfrak{S}}
\newcommand{\frakA}{\mathfrak{A}}

\newcommand{\bfk}{\mathbf{k}}

\newcommand{\frakf}{\mathfrak{f}}
\newcommand{\frake}{\mathfrak{e}}
\newcommand{\frako}{\mathfrak{o}}

\newcommand{\half}{\frac{1}{2}}

\newcommand{\Or}{\mathrm{O}}

\newcommand{\U}{\operatorname{U}}

\newcommand{\PGU}{\operatorname{PGU}}

\newcommand{\PGL}{\operatorname{PGL}}
\newcommand{\SU}{\operatorname{SU}}

\newcommand{\PSU}{\operatorname{PSU}}

\newcommand{\SO}{\operatorname{SO}}
\newcommand{\PSp}{\operatorname{PSp}}

\newcommand{\Bl}{\operatorname{Bl}}

\newcommand{\PSL}{\operatorname{PSL}}

\newcommand{\tr}{\textrm{tr}}

\newcommand{\Lef}{\operatorname{Lef}}

\newcommand{\bfF}{\mathbf{F}}

\newcommand{\beq}{\begin{equation}}
\newcommand{\eeq}{\end{equation}}

\renewcommand{\arraystretch}{1.15}

\theoremstyle{plain}
\newtheorem*{theorem*}{Theorem}

\tikzstyle{nodal}=[circle,draw,fill=black,inner sep=0pt, minimum width=4pt]
\tikzset{double distance = 2pt}

\begin{document}

\title[Automorphisms of del Pezzo surfaces in odd characteristic]
{Automorphisms of del Pezzo surfaces in odd characteristic}
\author{Igor Dolgachev}
\address{\hfill \newline 
Department of Mathematics \newline
University of Michigan \newline
525 East University Avenue \newline
Ann Arbor,
 MI 48109-1109 USA}
\email{idolga@umich.edu}

\author{Gebhard Martin}
\address{\hfill \newline
Mathematisches Institut  \newline
Universit\"at Bonn \newline
Endenicher Allee 60 \newline
53115 Bonn \newline
Germany}
\email{gmartin@math.uni-bonn.de}

\begin{abstract} We complete the classification of 
automorphism groups of del Pezzo surfaces over algebraically closed fields of odd positive 
characteristic.
\end{abstract}  

\maketitle

 \tableofcontents

\section*{Introduction} 
A del Pezzo surface of degree $d$ is a smooth projective algebraic surface $X$ 
with ample anti-canonical class $-K_X$ satisfying $K_X^2 = d$.  
It is known that $1\le d \le 9$, and $X\cong \bbP^2$ if $d =9 $. A del Pezzo surface of degree 
$8$ is isomorphic to a smooth quadric surface or to the blow-up of the projective plane in one point. 
All other del Pezzo surfaces of degree $d\le 7$, are isomorphic to the blow-up of $9-d$ points in $\bbP^2$ 
in general position.
Any non-degenerate smooth linearly normal surface in $\bbP^d$ of degree $d\ge 3$  is isomorphic to an anti-canonically embedded del Pezzo surface.

The group $\Aut(X)$ of automorphisms of a del Pezzo surface $X$ is a finite group if 
$d\le 5$ and a smooth algebraic group of positive dimension if $d\ge 6$.

The classification of automorphism groups of del Pezzo surfaces over an algebraically 
closed field of characteristic $0$ has been known for more than a hundred years. 
A modern exposition and the history can be found in \cite{CAG}. 
It is used in the classification of conjugacy classes of finite subgroups  of the Cremona group 
of the projective plane in \cite{DI}. The knowledge about possible groups of 
automorphisms of del Pezzo surfaces over algebraically closed fields of 
positive characteristic is essential for the extension of this classification to positive characteristic. 
Partial results in this direction can be found in 
\cite{Dolgachev1} and \cite{Dolgachev2}, and 
the classification of 
possible automorphism groups of cubic and quartic del Pezzo surfaces was accomplished in \cite{DD}. For the convenience of the reader, we recall this classification in Table \ref{tbl:autodp4} and Table \ref{tbl:cubics} in the Appendix.

Any rational surface $X$ with a group $G$ of automorphisms of order prime to the characteristic can be lifted 
to characteristic zero together with an action of $G$ on the lift \cite{Serre}. This reduces the classification of automorphism 
groups in positive characteristic $p$ to the classification of the groups of automorphisms of order divisible by $p$, which is the main part of this article,
and also to the analysis of pairs $(X,G)$ in characteristic zero that admit a good reduction modulo $p$, see Section \ref{sec: tamedeg2} and Section \ref{sec: tamedeg1}.
The classification of the groups of automorphisms of del Pezzo surfaces of degree 
larger than $4$ is rather easy, see Section \ref{S:3}. As the classification of automorphism groups of del Pezzo surfaces of degree $4$ and $3$ was achieved in \cite{DD},
we will complete the classification of automorphisms of del Pezzo surfaces by concentrating on the 
task of classifying automorphism groups of del Pezzo surfaces 
of degree $2$ and $1$. 

The cases of odd and even characteristic are drastically different and require different tools.
This is due to the fact that, in odd characteristic, the classification of automorphism groups of
del Pezzo surfaces of degree $2$ and $1$ essentially coincides with the classification 
of groups of projective automorphisms of smooth plane quartic curves or 
smooth genus $4$ curves of degree $6$ lying on a singular quadric. In the case of even characteristic no such relations exist.
 
 This is why we separate the classification into two parts. The present paper 
 deals with the case of odd characteristic and the subsequent paper \cite{DMeven} will deal with 
 the case of characteristic two. Our main result can be summarized as follows.

 \begin{theorem*}
A finite group $G$ is realized as the automorphism group $\Aut(X)$ of a del Pezzo surface $X$ of degree $2$ (resp. $1$) over an algebraically closed field $k$ of characteristic ${\rm char}(k) \neq 2$ if and only if $G$ is listed in Table \ref{tbl:autodp2} (resp. Table \ref{tbl:autodp1}) in the Appendix.
 \end{theorem*}

Table \ref{tbl:autodp2} (resp. Table \ref{tbl:autodp1}) also gives the conjugacy classes in $W(E_7)$ (resp. $W(E_8)$) of all elements of $\Aut(X)$ for all del Pezzo surfaces $X$ of degree $2$ (resp. degree $1$). Here, we use Carter's notation for these conjugacy classes, see \cite{Carter} and Section \ref{S:1}. The knowledge of these conjugacy classes is important in order to understand the conjugacy relation between the various automorphism groups of del Pezzo surfaces inside the Cremona group of rank $2$. 
 
\bigskip

\noindent \textbf{Acknowledgements:} The authors would like to thank Claudia Stadlmayr for helpful comments on a first draft of this article. 

 \newpage
 \section{Notation}
 We will use the following group-theoretical notations for the finite groups we encounter in the article.
 Throughout this article, $p$ is a prime and $q$ is a power of $p$.

\begin{itemize}
\item $C_n$ is the cyclic group of order $n$.
\item $\frakS_n$ and $\frakA_n$ are the symmetric and alternating groups on $n$ letters.
\item $Q_8$ is the quaternion group of order 8.
\item $D_{2n}$ is the dihedral group of order $2n$.
\item $n^k = (\bbZ/n\bbZ)^k$. In particular, $n = n^1 = \bbZ/n\bbZ$.
\item $\GL_n(q) = \GL(n,\bbF_q)$.
\item $\PGL_n(q) = \GL_n(q)/\bbF_q^*$. Its order is $N = q^{\half n(n-1)}(q^n-1) \cdots (q^2-1)$. 
\item $\SL_n(q) = \{g\in \GL_n(q):\det(g) = 1\}$. This is a subgroup of $\GL_n(q)$ of index $(q-1)$.
\item ${\rm L}_n(q) =\PSL_n(q)$ is the image of $\SL_n(q)$ in $\PGL_n(q)$. Its order is $N/(q-1,n)$. 
\item For odd $n$, $\Or_{n}(q)$ is the subgroup of $\GL_n(q)$ that preserves a non-degenerate quadratic form $F$.
\item For even $n$, $\Or_n^+(q)$ (resp. $\Or_n^-(q)$) is the subgroup of $\GL_n(q)$ that preserves a non-degenerate quadratic form $F$ of Witt defect $0$ (resp. $1$).
\item $\SO_{n}^{\pm}(q)$ is the subgroup of $\Or_{n}^{\pm}(q)$ of elements with determinant $1$. 
\item ${\rm PSO}_n^{\pm}(q)$ is the quotient of $\SO_{n}^{\pm}(q)$ by its center.
\item $\Sp_{2n}(q)$ is the subgroup of $\SL_q(2n)$ preserving the standard symplectic form on $\bbF_q^{2n}$.
Its order 
is $q^{n^2}(q^{2n-1}-1)\cdots (q^2-1)$.
\item $\PSp_{2n}(q) = \Sp_{2n}(q)/(\pm 1)$.
\item $\SU_n(q^2)$ is the subgroup of $\SL_n(q^2)$ of matrices preserving the hermitian form 
$\sum_{i=1}^nx_i^{q+1}$. 
Its order is 
$q^{\half n(n-1)}(q^n-(-1)^n)(q^{n-1}-(-1)^{n-1})\cdots (q^3+1)(q^2-1)$.
 We have $\SU_2(q^2) = \SL_2(q).$
\item $\PSU_n(q^2) = \SU_n(q^2)/C, $ where $C$ is a cyclic group of order $(q+1,n)$ of diagonal Hermitian matrices.
The simple group $\PSU_n(q^2)$ is denoted by 
$\U_n(q)$ in \cite{ATLAS}.
\item $\mathcal{H}_3(3)$ is the Heisenberg group of $3\times 3$ upper triangular matrices with entries in $\bbF_3$.
\item $A.B$ is a group that contains a normal subgroup $A$ with quotient group $B$.
\item $A:B$ is the semi-direct product $A\rtimes B$.
 \end{itemize}

\newpage

\section{Preliminaries}\label{S:1}
\subsection{The anti-canonical map} \label{SS:2.1} The main references for the known facts about del Pezzo 
surfaces to which we refer without proofs are \cite{Demazure} and \cite{CAG}. Throughout, we are working over an algebraically closed field of characteristic $p \neq 2$.

The linear system $|-K_X|$ on a del Pezzo surface $X$ 
of degree $d$ is of dimension $d$. It is base-point-free if $d\ne 1$ and defines an embedding 
$$j:X\hookrightarrow \bbP^d$$
if $d\ge 3$. The image is a non-degenerate smooth surface of degree $d$ in $\bbP^d$. 
If $d = 9$, the embedding
coincides with the third Veronese map $X\cong \bbP^2\hookrightarrow \bbP^9$. 

A del Pezzo surface of degree $d = 8$ is isomorphic to a minimal ruled surface 
$\bfF_0$ or $\bfF_1$. In all other cases, the image of $j$ is a projection of a Veronese surface from a linear span of $9-d$ points on the surface. 
This classical fact, due to del Pezzo, can be restated as the fact that any 
del Pezzo surface of degree $d\ge 3$ is isomorphic to a smooth quadric or obtained as the blow-up 
of $9-d$ points $p_1,\ldots,p_{9-d}$ in the plane. The latter description of del Pezzo surfaces 
extends to del Pezzo surfaces of the remaining
degrees $d = 1$ and $2$. Here, the points $p_1,\ldots,p_{9-d}$ must be in general position.

The linear system $|-K_X|$ is the proper inverse transform in the blow-up 
$\Bl_{p_1,\ldots,p_{9-d}}(\bbP^2)$ of the linear system of plane cubic curves 
$C^3(p_1,\ldots,p_{9-d})$ passing through the points $p_1,\ldots,p_{9-d}$.

For the following description of del Pezzo surfaces of degrees $1$ and $2$, we assume $p \neq 2$.
 
If $d = 2$, the anti-canonical linear system $|-K_X|$ is of dimension $2$. It defines a finite morphism 
of degree $2$ 
$$f:X\to \bbP^2$$
with branch locus a smooth quartic curve. This allows us to view $X$ as a hypersurface of degree $4$
in the weighted projective space $\bbP(1,1,1,2)$. Its equation is 
\beq\label{eq:dp2}
x_3^2+f_4(x_0,x_1,x_2) = 0
\eeq
and its branch locus is the smooth plane quartic curve $C = V(f_4(x_0,x_1,x_2))$. 

If $d = 1$, the linear system $|-K_X|$ is a pencil with one base point $\frako$. The 
linear system $|-2K_X|$ is of dimension $3$. It defines a finite morphism of degree $2$
$$f:X\to Q\subseteq \bbP^3,$$
where $Q$ is an irreducible quadratic cone. The cover is branched along a smooth curve 
of genus 4 cut out by a cubic. This allows us to consider $X$ as a hypersurface of degree 6 in the weighted 
projective space $\bbP(1,1,2,3)$ given by the equation
\[
w^2+f_6(x_0,x_1,x_2,x_3) = 0.
\]
This equation can be rewritten in the form
\beq\label{eq:dp1}
y^2+x^3+a_2(t_0,t_1)x^2+a_4(t_0,t_1)x+a_6(t_0,t_1) = 0,
\eeq
where $a_k$ is a binary form of degree $k$.

The double cover $f$ extends to a double cover 
$$f':Y = \Bl_{\frako}(X)\to \bfF_2,$$
where $\bfF_2 = \bbP(\calO_{\bbP^1}\oplus \calO_{\bbP^1}(-2))$ is a minimal rational ruled surface.
In the natural basis $(\frakf,\frake)$ of $\Pic(\bfF_2)$ with $\frakf^2 = 0, \frake^2 = -2$, the branch curve of $f'$ is the union of a smooth member of the linear system $|6\frakf+3\frake|$ and the exceptional section 
identified with $\frake$. 
Equation \eqref{eq:dp1} can be viewed as 
the Weierstrass model of the Jacobian elliptic fibration $Y\to \bbP^1$ defined by the proper 
transform of the anti-canonical pencil $|-K_X|$.

\subsection{Weyl groups}
An automorphism of a del Pezzo surface $X$ of degree $d$ acts naturally on the Picard group $\Pic(X)$ of isomorphism classes of invertible sheaves 
on $X$, or divisor classes modulo linear equivalence. 
Assume that $X\ne \bfF_0$, so that $X\cong \Bl_{p_1,\ldots,p_{9-d}}(\bbP^2)$. The group $\Pic(X)$ is a free abelian group generated by the divisor classes $e_i$ 
of the exceptional curves $E_i$ over the points $p_i$ and the divisor class $e_0$ such that 
$|e_0|$ defines the blowing down morphism $X\to \bbP^2$. 

 The basis $(e_0,e_1,\ldots,e_{9-d})$ is called 
a \emph{geometric basis} of $\Pic(X)$. It depends on the isomorphism 
$X\cong \Bl_{p_1,\ldots,p_{9-d}}(\bbP^2)$ since different point sets may lead to isomorphic blow-ups. 
If we fix one geometric basis, passing to another geometric basis defines an integer matrix of rank $10-d$. The set of all such matrices forms a subgroup $W$ of $\GL_{10-d}(\bbZ)$. 

The intersection product $\Pic(X)\times \Pic(X)\to \bbZ$ is a symmetric bilinear form 
$(x,y) \mapsto x\cdot y$ 
on $\Pic(X)$. It equips $\Pic(X)$ with the structure of a quadratic lattice. 
Since $e_0^2 = 1, e_i^2 = -1,\ i\ne 0$, the signature of the corresponding real quadratic space 
$\Pic(X)_{\mathbb{R}}:= \Pic(X)\otimes \mathbb{R}$ is equal to $(1,9-d)$. 
A geometric basis is an orthonormal basis for this inner product. This shows that the group $W$ is a subgroup 
of the orthogonal group $\Or(\Pic(X))$ of isometries of the quadratic lattice $\Pic(X)$.

The known behavior of the canonical class under a blow-up implies that 
\beq\label{canclass}
- K_X = 3e_0-\sum_{i=1}^{9-d}e_i,
\eeq 
and this formula is true for any geometric basis. This shows 
that $K_X$ is invariant with respect to the action of the group $W$, and hence 
$W$ can be identified with a subgroup of the orthogonal group $\Or(K_X^\perp)$ of 
the orthogonal complement to $\bbZ K_X$ in $\Pic(X)$.
The sublattice $K_X^\perp\subseteq \Pic(X)$ has a basis 
\beq\label{rootbasis}
 \alpha_0= e_0-e_1-e_2-e_3,\ \alpha_i = e_i-e_{i+1}.\  i = 1,\ldots,8-d.
 \eeq
Assume $d\le 6$. Computing the Gram matrix of this basis, we find that 
$K_X^\perp$ is isomorphic to the root lattice $\sfE_N$ of rank $N = 9-d$ of a simple Lie algebra of type 
$$A_2\oplus A_1, \quad A_4, \quad D_5, \quad E_6, \quad E_7, \quad E_8,$$
in the cases $d = 6,5,4,3,2,1$, respectively.
The basis \eqref{rootbasis} is a basis of simple roots. The subgroup of $\Or(\sfE_N)$ generated 
by the reflections in simple roots is the \emph{Weyl group} $W(\sfE_N)$ of the root lattice.
 Since any reflection in $\alpha_i$ extends to an isometry of the lattice 
$\sfI^{1,N}$ generated by $e_0,\ldots,e_{9-d}$, the group $W(\sfE_N)$ can be identified with the stabilizer subgroup 
$\Or(\sfI^{1,N})_{\bfk_N}$ of the vector $\bfk_N$ given in \eqref{canclass}. 
The homomorphism 
$$\Or(\sfI^{1,N})_{\bfk_N}\to \Or(\sfE_N)$$
is injective. If $N = 7, 8$, it is also surjective. If $N\ne 7, 8$, the image is a direct summand with complement of order 
$2$ generated by the symmetry of the Dynkin diagram of the root lattice $\sfE_N$ \cite[8.2]{CAG}.

Returning to our geometric situation, we see that the natural homomorphism
$$\Aut(X)\to \Or(K_X^\perp), \quad g\mapsto g^*,$$
composed with an isomorphism $\Or(K_X^\perp)\cong \Or(\sfE_{9-d})$ 
defines a homomorphism 
\beq\label{weylrepresentation}
\rho:\Aut(X)\to \Or(\sfE_{9-d}).
\eeq
Since the only automorphism of $\mathbb{P}^2$ that fixes four points in general position is the identity, we have the following theorem \cite[Corollary 8.2.40]{CAG}.

\begin{theorem}\label{thm:weylgroups1} The image of the homomorphism $\rho$ is contained in the Weyl group 
$W(\sfE_N)$. The homomorphism $\rho$ is injective for $d\le 5$.
\end{theorem}

In particular, via $\rho$, we can consider $\Aut(X)$ as a subgroup of $W(\sfE_N)$. We recall the well-known structure of the relevant Weyl groups. 

\begin{theorem}\label{thm:weylgroups2} The Weyl group of $E_{N}$ with $3 \leq N \leq 8$ can be described as follows:
\begin{itemize}
\item $W(\sfE_3) \cong \frakS_3\times 2$ of order $12$,
\item $W(\sfE_4) \cong \frakS_5$ of order $5!$, 
\item $W(\sfE_5) \cong 2^4:\frakS_5$ of order $2^4\cdot 5!$, 
\item $W(\sfE_6) \cong \Sp_4(3).2$ of order $2^3\cdot 3^2\cdot 6!$, 
\item $W(\sfE_7) \cong \Sp_6(2)\times 2$ of order $2^6\cdot 3^2\cdot 7!$, 
\item $W(\sfE_8) \cong 2.\Or_8^+(2)$ of order $2^7\cdot 3^3\cdot 5\cdot 8!$. 
\end{itemize}
\end{theorem}
In ATLAS \cite{ATLAS} notation $\Or_8^+(2) = {\rm GO}_8^+(2)$ contains a normal 
simple subgroup $\Or_8(2)$ of index $2$. The group $\Or_8^+(2)$ is the orthogonal group of the quadratic space 
$\bbF_2^{8}$ equipped with the quadratic form $x_1x_2+x_3x_4+x_5x_6+x_7x_8$.

It turns out that, even in arbitrary characteristic, not all 
cyclic subgroups of $W(\sfE_N)$ are realized by automorphisms of del Pezzo surfaces for $N \geq 3$. To determine which ones are realized, we will use the 
classification of conjugacy classes of elements $w$ of the Weyl groups. This classification can be found 
in \cite{Carter}. Table \ref{tbl:carter} gives the list of conjugacy classes 
of all Weyl groups in Carter's notation. According to Carter, these conjugacy classes are indexed by certain graphs that we call Carter graphs. 
The subscript in the name of a Carter graph indicates the number of vertices of this graph.

Recall that a root lattice $R$ is the orthogonal sum of root lattices of irreducible root systems of types $A_n,D_n,E_6,E_7,$ and $E_8$. If the name of a Carter graph is the name of a Dynkin diagram, say associated to the orthogonal sum of irreducible root lattices $R_1,\ldots,R_k$, then an element $w$ of the corresponding conjugacy class in $W(\sfE_N)$ preserves a sublattice of $\sfE_N$ isomorphic to $R_1 \perp \ldots \perp R_k$. Moreover, this $w$ acts on each summand as the Coxeter element of $W(R_i)$ and as the identity on the orthogonal complement of $R_1 \perp \ldots \perp R_k$.
For example, the notation $A_{i_1}+\cdots + A_{i_k}$ (we write $kA_n$ if $i_1=\cdots=i_k = n$) for a conjugacy class of an element $w\in W(\sfE_N)$  
 means that $w$ leaves invariant a sublattice of $\sfE_N$ isomorphic to the orthogonal sum
 $A_{i_1}\perp\ldots\perp A_{i_k}$ of root lattices of type $A_{i_s}$,
 and it acts on each summand as an 
 element of order $i_s+1$ in $W(A_{i_s})\cong \frakS_{i_s+1}$. 

 The Carter graphs that contain cycles are named $\Gamma(a_i)$, where $\Gamma$ is the name of a Dynkin diagram. These graphs correspond to certain conjugacy classes in $W(\Gamma)$ that do not leave any non-trivial sublattices invariant. We refer the reader to \cite[Table 2]{Carter} for a detailed description of these graphs.

The characteristic polynomials given in Table \ref{tbl:carter} allow us to compute the trace $\tr_2(w)$ of 
$w\in W(\sfE_N)$ acting on the \'etale $l$-adic cohomology $H^2(X,\bbQ_l) \cong \Pic(X)_{\bbQ_l}$. 
Since $w$ acts trivially on $H^0(X,\bbQ_l)$ and $H^4(X,\bbQ_l)$, and $H^i(X,\bbQ_l) = 0$ for odd 
$i$, we can use the Lefschetz fixed point formula 
\beq\label{lefschetz}
\Lef(g^*)= 2+\tr_2(\rho(g)) = e(X^{g})
\eeq
to compute the trace from the Euler-Poincar\'e characteristic of the fixed locus $X^g$ of a tame automorphism $g\in \Aut(X)$.

If $p$ divides the order of $g$, we could instead use Saito's generalized Lefschetz fixed-point formula 
\cite{Saito}. However, we are able to avoid an application of this rather technical tool and compute the traces by 
using geometric methods. This will allow us to determine the conjugacy classes of wild elements $\rho(g)\in W(\sfE_N)$.

\begin{remark}\label{atlas} One can also derive the description of the conjugacy classes of elements of the Weyl groups 
$W(\sfE_{9-d})$ with $ d = 1,2,3$ from \cite{ATLAS}. The disadvantage of the notation there is that it is not uniform for different Weyl groups.
The conjugacy classes in the Weyl group $W(\sfE_6)$ (resp. $W(\sfE_7)$, resp. $W(\sfE_8)$) can be read off from the notation of the conjugacy classes in the group 
$S_4(3)$ (resp. $S_6(2)$, resp. $O_8^+(2)$) in \cite{ATLAS}. For example, a translation from Carter's notation to ATLAS notation for the group
$W(\sfE_6)$ can be found in \cite[Table 6]{DD}.
\end{remark}

\subsection{The Geiser and Bertini involutions} 
It follows from the description of a del Pezzo surface of degree $2$ (resp. $1$) as a double cover that $\Aut(X)$ contains a central element 
$\gamma$ (resp. $\beta$) realized by the negation of $x_3$ (resp. $y$). It is called the \emph{Geiser involution} (resp. 
the \emph{Bertini involution}) of $X$. 

The image of $\gamma$ (resp. $\beta$) under $\rho$ in \eqref{weylrepresentation} is 
the unique non-trivial central element $w_0$ of $W(\sfE_N)$ for $N = 7,8$. It is the unique element of $W(\sfE_N)$ 
of maximal length when written as the reduced product of reflections in simple roots. It acts as $-\id_{\sfE_N}$ if $N = 7,8$.

The action of $g\in \Aut(X)$ fixes the branch curve $Q$ of the double cover $f$ pointwise. Since 
the curve is embedded in $\bbP^2$ (resp. $\bbP(1,1,2)$) by the canonical linear system, the action 
on $Q$ extends to an automorphism of $\bbP^2$ (resp. $\bbP(1,1,2)$). This defines a homomorphism
$$\psi:\Aut(X) \to \Aut(\bbP^2)
\textrm{ \hspace{2mm} (resp. } \psi:\Aut(X) \to \Aut(\bbP(1,1,2)).$$
We have 
$$\Ker(\psi) = (\gamma) \textrm{ \hspace{2mm} (resp. $\Ker(\psi) =  (\beta)$}),
$$
since $Q$ is non-degenerate.
 
\subsection{Wild finite subgroups of $\Aut(\bbP^1)$ and $\Aut(\bbP^2)$}\label{SS:2.4}

A finite subgroup of order $n$ of an algebraic group over a field $\Bbbk$ of characteristic $p > 0$ is called 
\emph{wild} (resp. \emph{tame} or \emph{$p$-regular}) if $p\vert n$ (resp. $(n,p) = 1$). An element of finite order is wild (tame) if it generates 
a wild (tame) cyclic group.

The following proposition follows immediately from Theorem \ref{thm:weylgroups1} and Theorem \ref{thm:weylgroups2}.

\begin{proposition} Let $g$ be a wild automorphism of order $p$ of a del Pezzo surface of degree $1$ or $2$. 
Then $p\in \{2,3,5,7\}$.
\end{proposition}

We will show later that $p\ne 5,7$ (resp. $p\ne 7$) in the case of del Pezzo surfaces of degree $2$ (resp. $1$).
It is known that there are no wild automorphisms of del Pezzo surfaces of odd order $p^2$ \cite[Theorem 8]{Dolgachev1}.

The image of a wild group $G$ of automorphisms of a del Pezzo surface of degree 2 
under the homomorphism $\psi$ is a wild subgroup of $\Aut(\bbP^2)$. If $G$ is a wild group of automorphisms 
of a del Pezzo surface of degree $1$, then $G$ acts on the pencil $|-K_X|$ and, if the image of a wild 
element of $G$ is non-trivial, it 
defines a wild subgroup of $\Aut(\bbP^1)$.

It is known that any finite subgroup $G$ of $\GL_n(\Bbbk)$ is isomorphic to a finite subgroup of $\GL_n(q)$ for some 
$q = p^m$. The classification of wild finite subgroups of $\Aut(\bbP^1)\cong \PGL_2(\Bbbk)$ can be found in \cite[Chapter 3, \S 6]{Suzuki}. We summarize it here for the convenience of the reader.

\begin{theorem}\label{suz} Let $G$ be a proper wild subgroup of $\PGL_2(\Bbbk)$. Then $G$ is isomorphic to one of the following groups.
\begin{itemize}
\item[(1)] The group $G_{\xi,A}$ of affine transformations $x\mapsto \xi^tx+a$, where $a$ is an element of a 
finite subgroup $A$ of the additive group of $\Bbbk$ containing 1 and $\xi$ is a root of unity such that $\xi A = A.$
\item[(2)]  A dihedral group of order $2n$ with $n$ is odd if $p = 2$.
\item[(3)] ${\rm L}_2(5)\subset {\rm L}_2(9)$ if $p = 3$. 
\item[(4)] ${\rm L}_2(q)$ or $\PGL_2(q)$.
\end{itemize}
\end{theorem}

The classification of wild subgroups of $\Aut(\bbP^2)\cong \PGL_3(\Bbbk) \cong \PSL_3(\Bbbk)$ is also known. 
 It follows from the classification of conjugacy classes of subgroups of the groups 
 $\PSL_3(q)$  \cite{Bloom}, \cite{Mitchell}. Again, we summarize it here for the convenience of the reader.

\begin{theorem}\label{bloom}  Assume $p > 2$. Let $G$ be a finite subgroup of 
$\PGL_3(k)$.
Then $G$ is conjugate to one of the following groups.
\begin{enumerate}
\item  $\PGL_3(q)$ or ${\rm L}_3(q)$.
\item $\PGU_3(q^2)$ or $\PSU_3(q^2)$.
\item A group containing type (1) with $3 \mid (q-1)$ as a normal subgroup of index $3$.
\item A group containing type (2) with $3 \mid (q+1)$ as a normal subgroup of index $3$.
\item $\PGL_2(q)$ or ${\rm L}_2(q)$ with $q \neq 3$.
 \item If $p \neq 5$: ${\rm L}_2(5) \cong \frakA_5$.
 \item If $p \neq 7$: ${\rm L}_2(7)$.
 \item ${\rm L}_2(9) \cong \frakA_6$.
\item If $p = 5$: A group containing $\frakA_6$ of index $2$.
 \item If $p = 5$: $\frakA_7$.
\item A group containing a normal cyclic tame subgroup of index $\le 3$.
\item A group containing a diagonal normal subgroup $H$ such that $G/H$ is isomorphic to a subgroup of 
$\frakS_3$.
\item A group whose inverse image $\tilde{G}$ in $\SL_3(\Bbbk)$ has a normal elementary abelian $p$-subgroup 
$H$ such that $\tilde{G}/H$ is a subgroup of  $\GL_2(q)$.
\item If $p \neq 3$: The Hessian group $3^2:\SL_2(3)$ of order $216$ or its subgroups containing $3^2$.
\item If $p \neq 3$: The group $3^2:Q_8$ or its subgroups containing $3^2$.
\end{enumerate}
\end{theorem}

\section{Del Pezzo surfaces of degree $\ge 3$}\label{S:3}
For the convenience of the reader, we recall here the classification of automorphisms groups of del Pezzo surfaces of degree at least $3$.

\subsection{Del Pezzo surfaces of degree $\ge 5$}
The computation of automorphism groups of del Pezzo surfaces of degree $\ge 5$ is 
characteristic free and can be found, for example, in \cite[Chapter 8]{CAG}:

\subsubsection{Degree $9$}
If $d = 9$, then $X\cong \bbP^2$ and $|K_X| = |\calO_{\bbP^2}(3)|$. The embedding 
$\nu:\bbP^2\hookrightarrow \bbP^9$ coincides with a Veronese embedding. The group of automorphisms
of $X$ coincides with the projective linear group $\PGL_3(\Bbbk)$.

\subsubsection{Degree $8$}
If $d = 8$, then $X\cong \bbP^1\times \bbP^1$, or $X\cong \bfF_1$, the blow-up 
of one point in $\bbP^2$. 

If $X\cong \bbP^1\times \bbP^1$, the anti-canonical map is given by the linear system 
$|-K_X| = |\calO_{\bbP^1\times \bbP^1}(2,2)|$.  It coincides with the Segre--Veronese map and embeds 
$\bbP^1\times \bbP^1$ into $\bbP^8$. 
The group $\Aut^0(X)$ coincides with $\PGL_2(\Bbbk)\times \PGL_2(\Bbbk)$, and the quotient group 
$\Aut(X)/\Aut^0(X)$ is of order 2, generated by switching the factors of the product 
$\bbP^1\times \bbP^1$. 

If $X\cong \bfF_1 = {\rm Bl}_{p_1}(\mathbb{P}^2)$, the anti-canonical system is $|-K_X| = |\calO_{\bbP^2}(3)-p_1|$. 
The anti-canonical linear system defines a map that embeds $X$ into $\bbP^8$. 
Its image is equal to the projection 
of the Veronese surface of degree $9$ from a point lying on it. The group $\Aut(X)$ is isomorphic to 
the stabilizer subgroup of $\{p_1\}$ in $\PGL_3(\Bbbk) = \Aut(\bbP^2)$. It is a solvable algebraic 
subgroup of $\PGL_3(\Bbbk)$ of dimension $6$.

\subsubsection{Degree $7$}

If $d = 7$, then $X\cong \Bl_{p_1,p_2}(\bbP^2)$ and $|-K_X| = |\calO_{\bbP^2}(3)-p_1-p_2|$. 
The anti-canonical linear system defines a map that embeds $X$ into $\bbP^7$. The image is the projection 
of the Veronese surface of degree 9 from two points lying on it.  The group $\Aut(X)$ is isomorphic to 
the stabilizer subgroup of $\{p_1,p_2\}$ in $\PGL_3(\Bbbk) = \Aut(\bbP^2)$. It is a solvable algebraic subgroup of $\PGL_3(\Bbbk)$ of dimension $4$. The quotient group $\Aut(X)/\Aut^0(X)$ is of order $2$. 
It is generated by a projective involution that switches the points $p_1,p_2$.

\subsubsection{Degree $6$}

If $d = 6$, then $X\cong \Bl_{p_1,p_2,p_3}(\bbP^2)$ and $|-K_X| = |\calO_{\bbP^2}(3)-p_1-p_2-p_3|$, where the points 
$p_1,p_2,p_3$ are not collinear. The group $\Aut^0(X)$ is isomorphic 
to the two-dimensional torus $\bbG_{m,\Bbbk}^2$. In appropriate projective coordinates where
$p_1 = [1,0,0], p_2 = [0,1,0],$ and $ p_3 = [0,0,1]$, the group  $\Aut^0(X)$ consists of transformations 
$(x_0,x_1,x_2)\mapsto (\lambda x_0,\mu x_1,\gamma x_2)$. The quotient group is isomorphic to 
$\frakS_3\times 2$. It is generated by projective transformations which permute the points and the 
standard quadratic Cremona transformation $(x_0,x_1,x_2)\mapsto (x_1x_2,x_0x_2,x_0x_1)$.

\subsubsection{Degree $5$}
If $d = 5$, then $X\cong \Bl_{p_1,p_2,p_3,p_4}(\bbP^2)$ and $|-K_X| = |\calO_{\bbP^2}(3)-p_1-p_2-p_2-p_4|$.
 In this case $\Aut^0(X) = \{1\}$ and $\Aut^0(X) \cong \frakS_5$. The group is generated by its subgroups
 of projective transformations that permute the four points and also a quadratic Cremona transformation 
 of order $5$
 \beq\label{cremonadP5}
(x_0,x_1,x_2) \mapsto (x_0(x_2-x_1),x_2(x_0-x_1),x_0x_2).
\eeq
where we use projective coordinates such that $p_1,p_2,p_3$ are as in the previous case, and $p_4 = [1,1,1]$.

\subsection{Quartic del Pezzo surfaces} \label{S:3.2}
This is the first case where the classification of automorphism groups depends on the characteristic $p$ of the ground field. We assume $p \neq 2$ and refer the reader to \cite{DMeven} for the case of $p = 2$.

A quartic del Pezzo surface $X$ is isomorphic to the blow-up $\Bl_{p_1,\ldots,p_5}(\bbP^2)$, where no 
three points are collinear. The anti-canonical linear system 
$|-K_X| = |\calO_{\bbP^2}(3)-p_1-p_2-p_3-p_4-p_5|$ embeds $X$ into $\bbP^4$ as a complete intersection of 
two quadrics.

If $p = 0$, the classification of automorphism groups can be found in \cite[8.6]{CAG} and if $p \neq 0$, it was accomplished in \cite[\S 3]{DD}. Since $p \neq 2$, the two quadrics can be simultaneously 
diagonalized so that $X$ is given by equations 
\beq\label{quartcdP}
t_1^2+t_2^2+t_3^2+at_4^2 = t_0^2+t_2^2+bt_3^2+t_4^2 = 0,\
\eeq
where the binary form $\Delta = uv(u-v)(u-av)(bu-v)$ has no multiple roots.
 
 The automorphism group contains a normal subgroup $H$, isomorphic 
to $2^4$ and generated by the transformations $(t_0,\ldots,t_4)\mapsto (\pm t_0,\ldots,\pm t_4)$. The quotient group 
$\Aut(X)/H$ is isomorphic to the subgroup $G$ of $\PGL_2(\Bbbk) \cong \Aut(\bbP^1)$ that leaves the set $V(\Delta)$ invariant.
The classification is summarized in Table \ref{tbl:autodp4} in the Appendix. There, the first column refers to the values of the parameters $a$ and $b$ in Equation \eqref{quartcdP} above. The conjugacy classes of elements of $\Aut(X)$ can be obtained by combining \cite[Table 2]{DD} and \cite[Table 5]{Carter}.

\subsection{Cubic surfaces} \label{S:3.3}
A del Pezzo surface of degree $3$ is isomorphic to the blow-up of 6 points 
$p_1,\ldots,p_6$, where no three points are collinear and not all of them lie on a conic.
The anti-canonical linear system $|-K_X| = |\calO_{\bbP^2}(3)-p_1-p_2-p_3-p_4-p_5-p_6|$ embeds $X$ into $\bbP^3$. Its image is a cubic surface.

As we remarked in the introduction, the classification of automorphism groups of smooth cubic surfaces in characteristic 
$0$ has been known essentially since the 19th century (see the most recent exposition in \cite[9.5]{CAG}). In the case of positive characteristic it can be found in \cite{DD}. We present the results of this 
classification in Table \ref{tbl:cubics} in the Appendix. Our table is obtained by translating \cite[Table 7]{DD} from ATLAS to Carter notation. We also give the name of the corresponding type in \cite[Table 9.6]{CAG} for the convenience of the reader.

\section{Del Pezzo surfaces of degree $2$}
In this section, we are concerned with the classification of automorphism groups of del Pezzo surfaces of degree $2$. Recall from Section \ref{SS:2.1} that every del Pezzo surface $X$ of degree $2$ is a double cover of $\mathbb{P}^2$ branched over a smooth quartic $C$. Moreover, we have the product decomposition $\Aut(X) \cong 2 \times \Aut(C)$, which follows from the corresponding decomposition of $W(\sfE_7)$ (see Theorem \ref{thm:weylgroups2}), so classifying automorphism groups of del Pezzo surfaces of degree $2$ is equivalent to classifying automorphism groups of smooth plane quartic curves.

\subsection{List of possible groups}
We know from Section \ref{SS:2.4} that a wild automorphism of $X$ can occur only if $p\in \{3,5,7\}$. 
 The next lemma eliminates the cases $p = 5$ and $p = 7$.

\begin{lemma}\label{lem:wildcyclic} 
Assume that an element $g$ of order $p$ acts non-trivially on a del Pezzo surface $X$ of degree $2$. Then, $p = 3$. 
\end{lemma}

\begin{proof} Since $p$ is odd, $g$ leaves invariant the 
plane quartic $C$. There are two possible Jordan forms for a cyclic linear automorphism of $\mathbb{P}^2$ of
order $p$, namely $(x_0,x_1,x_2) \mapsto (x_0+x_1,x_1,x_2)$ and $(x_0,x_1,x_2)\mapsto (x_0+x_1,x_1+x_2,x_2)$. 
Since $g(C) = C$, and the cyclic group $(g)$ admits no non-trivial characters, the equation of $C$ is $g$-invariant. 

By \cite[Section 10.3]{Wehlau}, the invariant ring in the first case is
$$
\Bbbk[x_0,x_1,x_2]^{(g)} = \Bbbk[N(x_0),x_1,x_2],
$$
where $N(x_0) = x_0(x_0+x_1)\cdots (x_0 + (p-1)x_1)$. If $p \neq 3$, the degree of 
$N(x_0)$ is larger than $4$, hence $C$ is a union of lines through $[1,0,0]$, which is absurd.
 
By \cite[Theorem 10.5]{Wehlau}, the invariant ring in the second case is
$$
\Bbbk[x_0,x_1,x_2]^{(g)} = \Bbbk[N(x_0),x_1^2 - 2x_0x_2 - x_1x_2,x_2,
\{\tr(x_1^bx_0^c)\}_{0 \leq b \leq 1, 0 \leq c \leq p-1}],
$$
where $N(x_0) = \prod_{i=0}^{p-1}(x_0 + ix_1 + \frac{i^2 - i}{2}x_2)$ and $\tr(x_1^bx_0^c) = \sum_{i=0}^{p-1} (x_1 + ix_2)^b(x_0 + ix_1 + \frac{i^2 - i}{2}x_2)^c$. Recall that 
$$
\sum_{i = 0}^{p-1} i^n = 
\begin{cases}
-1 \text{ if } (p-1) \mid n \\
0 \text{ else}
\end{cases}
$$
Hence, if $p > 3$, all the polynomials $\tr(x_1^bx_0^c)$ with $b + c \leq 4$ have multiplicity at 
least $c - 2$ at $[1,0,0]$, so every element of $$\Bbbk[x_0,x_1,x_2]^{C_p}_4 = \Bbbk[x_1^2 - 2x_0x_2 - x_1x_2,x_2,\{\tr(x_1^bx_0^c)\}_{0 \leq b \leq 1, 0 \leq c \leq p-1, b+c \leq 4}]_4$$ defines a quartic which is singular at $[1,0,0]$. Therefore, we must have $p = 3$. 
\end{proof}

\begin{corollary} \label{cor: order57}
There exists no smooth plane quartic with an automorphism of order $5$.
\end{corollary}
\begin{proof}
By Lemma \ref{lem:wildcyclic}, this holds if $p = 5$. For $p \neq 5$, we can lift the curve to characteristic $0$ together with its automorphism of order $5$. But in characteristic $0$, there are no smooth quartics with automorphisms of order $5$, a contradiction.  
\end{proof}

Since we also know that $\Aut(C)$ does not contain automorphisms of order $p^2$, we obtain the following.

\begin{corollary} \label{cor: possibleorders} 
Let $G$ be a finite group that acts faithfully on a smooth plane quartic. Then, the following hold.
\begin{enumerate}
    \item The only prime divisors of $|G|$ are $2,3,$ and $7$.
    \item If $G$ is wild, then $p = 3$.
\end{enumerate}
\end{corollary}
\begin{proof}
If $G$ is wild, then it contains $C_p$, so $p = 3$ by Lemma \ref{lem:wildcyclic}. Since $G$ embeds into $W(E_7)$ via $\rho$ and the only prime divisors of $|W(E_7)|$ are $2,3,5,$ and $7$ by Theorem \ref{thm:weylgroups2}, the same holds for $|G|$. By Corollary \ref{cor: order57}, the divisor $5$ does not occur, and the claim follows. 
\end{proof}

The restrictions $|G|$ obtained thus far allow us to show that the list of groups in Theorem \ref{bloom} that can leave a smooth plane quartic invariant is relatively short.

\begin{theorem} \label{thm: possiblegroups}
Assume $p \neq 2$. Let $G$ be a finite subgroup of $\PGL_3(\Bbbk)$ that leaves invariant a smooth plane quartic curve. Then, $G$ is conjugate to one of the following groups:
\begin{enumerate}
    \item Only if $p = 3$: The group $\PSU_3(9)$.
    \item Only if $p \neq 7$: The group ${\rm L}_2(7)$.
    \item A subgroup of $(\Bbbk^{\times})^2 : \frakS_3$, where $(\Bbbk^{\times})^2$ acts diagonally and $\frakS_3$ acts via permutations.
     \item Only if $p = 3$: A subgroup of $\Bbbk^2 : \GL_2(\Bbbk)$, where $\Bbbk^2$ acts as $x_0 \mapsto x_0 + \lambda x_1 + \mu x_2$ and $\GL_2(\Bbbk)$ acts linearly on $x_1$ and $x_2$. 
\end{enumerate}
\end{theorem}
\begin{proof}
In the following, we study the items in Theorem \ref{bloom} and exclude most of them using Corollary \ref{cor: possibleorders}. Recall from Theorem \ref{thm:weylgroups2} that
$$
|W(\sfE_7)| = 2^{10} \cdot 3^4 \cdot 5 \cdot 7$$
and that, by Corollary \ref{cor: possibleorders}, the only prime divisors of $|G|$ are $2$, $3$, and $7$.
\begin{enumerate}
    \item[(1)-(5)] All these groups are wild, hence $p = 3$ and $q$ is a power of $3$. Then, the only group whose order divides $|W(\sfE_7)|$ and which does not have $5$ as a prime factor is $\PSU_3(9)$. 
    \item[(6)-(10)] Here, all groups except ${\rm L}_2(7)$ have $5$ as a prime factor.
    \item[(11)-(12)] The group in (11) can be diagonalized over $\Bbbk$, so (11) becomes a special case of (12).
    \item[(13)] The explicit description we give is taken from \cite[Section 7]{Bloom}.
    \item[(14)-(15)] Over $\Bbbk$, (15) is a special case of (14). All these groups contain a subgroup of the form $3^2$. Since $p \neq 3$, we can diagonalize this subgroup. It is elementary to check that there are no smooth quartics that are invariant under a diagonal action of $3^2$.
\end{enumerate}
\end{proof}

In the following sections, we analyze these four types of groups in detail and check which of them are realized as automorphism groups of smooth plane quartics. We will focus on the wild groups first and treat the tame groups afterwards in Section \ref{sec: tamedeg2}.

\subsection{The groups $\PSU_3(9)$ and 
${\rm L}_2(7)$}

It is well-known that, over $\mathbb{C}$, the Klein quartic curve defined by the equation
\beq\label{klein}
x_0^3x_1 + x_1^3x_2 + x_2^3x_0 = 0.
\eeq
is the unique (up to isomorphism) smooth plane quartic curve with group of automorphisms of maximal 
possible order equal to $168$. Its group of automorphisms is isomorphic to the simple group ${\rm L}_2(7)$. Less known is that it is isomorphic to two members of the pencil of plane quartic curves
\beq
x^4+y^4+z^4+t(x^2y^2+y^2z^2+y^2z^2) = 0,
\eeq
corresponding to the parameters $t = \frac{3}{2}(-1\pm \sqrt{-7})$ (see \cite[6.5.2]{CAG}). 
The change of coordinates realizing this isomorphism is given by 
$$(x_0,x_1,x_2) \mapsto (x+ay+bz,ax+by+z,bx+y+az),$$
where $a= 1+\zeta_7+\zeta_7^2+\zeta_3+\zeta_7^5, b = \zeta_7^2+\zeta_7$ and 
$\zeta_7$ is a primitive $7$-th root of unity \cite[p.56]{Elkies}.

In our case, when $p = 3$, we can still use this transformation and obtain an isomorphism between the Klein curve
\eqref{klein} and the Fermat quartic given by the equation
\beq\label{fermat}
    x_0^4 + x_1^4 + x_2^4 = 0.
\eeq

Since automorphisms of order $7$ are tame, the proof of \cite[Lemma 6.5.1]{CAG} shows the following.

\begin{lemma} \label{lem:quarticc7}
Let $C$ be a smooth plane quartic with an automorphism of order $7$. Then, $C$ is isomorphic to the Klein quartic
$$
x_0^3x_1 + x_1^3x_2 + x_2^3x_0 = 0.
$$
\end{lemma}

\begin{corollary} \label{cor: FermatKlein}
Let $C$ be a smooth plane quartic. Then, the following are equivalent:
\begin{enumerate}
    \item $C_7 \subseteq \Aut(C)$.
    \item ${\rm L}_2(7) \subseteq \Aut(C)$.
    \item $C$ can be defined by the equation
    \begin{equation} \label{eqndp2PSU39}
    x_0^3x_1 + x_1^3x_2 + x_2^3x_0 = 0.
    \end{equation}
\end{enumerate}
\smallskip

If $p \neq 3$, then (1) -- (3) are also equivalent to the following:
\begin{enumerate}
    \item[(4)] ${\rm L}_2(7) = \Aut(C)$.
\end{enumerate}
\smallskip

If $p = 3$, then (1) -- (3) are also equivalent to the following:
\begin{enumerate}
    \item[(5)] $\PSU_3(9) \subseteq \Aut(C)$.
    \item[(6)] $\PSU_3(9) = \Aut(C)$.
    \item[(7)] $C$ can be defined by the equation
    $$
    x_0^4 + x_1^4 + x_2^4 = 0
    $$
\end{enumerate}
In particular, for all $p$, the curve $C$ satisfying these conditions is unique and it exists if and only if $p \neq 7$.
\end{corollary}
\begin{proof}
By Lemma \ref{lem:quarticc7}, we have $(1) \Rightarrow (3)$. The implication $(3) \Rightarrow (2)$ holds over $\mathbb{C}$. This yields the corresponding implication in positive characteristic via reduction modulo $p$. The implication $(2) \Rightarrow (1)$ is clear. Note that the equation in $(3)$ defines a smooth quartic if and only if $p \neq 7$.

If $p \neq 3$, then Theorem \ref{thm: possiblegroups} shows that ${\rm L}_2(7)$ is maximal among all subgroups of $\PGL_3(\Bbbk)$ preserving a smooth quartic, hence $(2) \Rightarrow (4)$. The converse is clear.

If $p = 3$, then, as explained above, we have $(3) \Rightarrow (7)$. The Fermat quartic admits a faithful action of $4^2$, hence $\Aut(C)$ is strictly larger than ${\rm L}_3(7)$. By Theorem \ref{thm: possiblegroups}, this implies $\PSU_3(9) = \Aut(C)$. Hence, $(7) \Rightarrow (6)$. The remaining implications $(6) \Rightarrow (5)$ and $(5) \Rightarrow (1)$ are clear.
\end{proof}

\begin{remark} \label{rem: Hessian}
    The Fermat quartic curve in characteristic $3$ can be characterized among smooth plane quartics by the property that it has infinitely many 
inflection points, or equivalently, its Hessian is identically zero \cite[Proposition 3.7]{Pardini}.
\end{remark}

\subsection{Wild subgroups of $(\Bbbk^{\times})^2 : \mathfrak{S}_3$}

In this section, we classify the wild subgroups of $(\Bbbk^{\times})^2 : \mathfrak{S}_3$ that leave invariant smooth plane quartics. So, we have to classify smooth quartics that are invariant under a cyclic permutation of coordinates in characteristic $p = 3$.

\begin{lemma} \label{lem:3C}
Let $C$ be a smooth plane quartic in characteristic $p = 3$. Then, $C$ is invariant under the automorphism
 $g: (x_0,x_1,x_2) \mapsto (x_1,x_2,x_0)$ if and only if $C$ is given by an equation of the form
\begin{eqnarray*}
    &a(x_0^4 + x_1^4 + x_2^4) + b(x_0^3x_1 + x_1^3x_2 + x_2^3x_0) + c(x_0x_1^3 + x_1x_2^3 + x_2x_0^3) \\
    +&d(x_0^2x_1^2 + x_1^2x_2^2 + x_2^2x_0^2) + 
    e(x_0x_1x_2^2 + x_1x_2x_0^2 + x_2x_0x_1^2) = 0
\end{eqnarray*}
    with $a,b,c,d,e \in \Bbbk$. The element $g$ belongs to the conjugacy class of Type $2A_2$ in $W(\sfE_6)$.
\end{lemma}

\begin{proof}
Computing the equation of $C$ is straightforward. To compute the conjugacy class of $g$, we can argue as follows:
Since $C$ has $28$ bitangents and $g$ has order $3$, the automorphism of the del Pezzo surface $X$ induced by 
$g$ preserves a $(-1)$-curve, hence $g$ is induced by an automorphism $g'$ of a cubic surface and $X$ is 
obtained by blowing up a fixed point of $g'$. As $g$ has only one fixed point on $\mathbb{P}^2$, 
the automorphism $g'$ also has only one fixed point. By \cite[Theorem 10.4]{DD} this implies that the 
conjugacy class of $g'$ is of Type $2A_2$.  
\end{proof}

It turns out that every smooth quartic which is invariant under a cyclic permutation of coordiantes is projectively equivalent to one that is invariant under arbitrary permutations of coordinates. We remark that this happens in characteristic $0$ as well (see \cite[Proof of Theorem 6.5.2]{CAG}).

\begin{lemma} \label{lem:S3}
Let $C$ be a smooth plane quartic in characteristic $p = 3$. Assume that $C$ is invariant under $g: (x_0,x_1,x_2) \mapsto (x_1,x_2,x_0)$. Then, $C$ is invariant under a subgroup of $\PGL_3(\Bbbk)$ which is conjugate to the group $\mathfrak{S}_3$ of permutation matrices. In particular,
 $C$ admits an equation of the form
\begin{eqnarray} \label{eqndp2S3}
    &a(x_0^4 + x_1^4 + x_2^4) + b(x_0^3x_1 + x_1^3x_2 + x_2^3x_0 +x_0x_1^3 + x_1x_2^3 + x_2x_0^3) \\
    +&c(x_0^2x_1^2 + x_1^2x_2^2 + x_2^2x_0^2) + 
    d(x_0x_1x_2^2 + x_1x_2x_0^2 + x_2x_0x_1^2) = 0 \nonumber
\end{eqnarray}
\end{lemma}
\begin{proof}
It suffices to prove that the statement holds for a general $C$ in the family described in Lemma \ref{lem:3C}. With notation as in that lemma, the quartics where $b = c$ are invariant under $\mathfrak{S}_3$. The centralizer $Z$ of $g$ in $\GL_3(\Bbbk)$ consists of matrices of the form
$$
\left(
\begin{matrix}
\alpha & \beta & \gamma \\
\gamma & \alpha & \beta \\
\beta & \gamma & \alpha
\end{matrix}
\right)
$$
with $\alpha + \beta + \gamma \neq 0$, and the centralizer of the full group $\mathfrak{S}_3$ is the subgroup where $\beta = \gamma$. So, if we conjugate the equations where $b = c$ with elements of $Z$ with $\beta \neq \gamma$, 
we obtain a general $C$ in the family. Hence, $\mathfrak{S}_3 \subseteq \Aut(C)$ holds for a general $C$, hence for 
every $C$, in the family of Lemma \ref{lem:3C}.
\end{proof}

Now, the Fermat quartic \eqref{fermat} is a member of the family of 
quartics that occur in Lemma \ref{lem:S3}. The following lemma gives us a criterion to detect it.

\begin{lemma} \label{lem: detectfermat}
Assume $p = 3$. A smooth plane quartic $C$ given by an equation as in Lemma \ref{lem:3C} 
satisfies $\Aut(C) = \PSU_3(9)$ if and only if $c = d = 0$ and $b\ne a$.
\end{lemma}
\begin{proof}
Computing the Hessian of this curve, we find that it is equal to zero if and only if 
$c=d = 0$. We conclude using Remark \ref{rem: Hessian}.
\end{proof}

\begin{lemma} \label{lem:S4}
Let $C$ be a smooth plane quartic in characteristic $p = 3$. Assume that $\Aut(C)$ is conjugate to a 
subgroup of $(\Bbbk^{\times})^2 : \mathfrak{S}_3$ and $\mathfrak{S}_3 \subsetneq \Aut(C)$. Then, $C$ can be defined by an equation of the form
\beq\label{eqndp2S4}
x_0^2x_1^2 + x_1^2x_2^2 + x_0^2x_2^2 + a(x_0^4 + x_1^4 + x_2^4) = 0
\eeq
with $a \in \Bbbk \setminus \mathbb{F}_3$.
Conversely, a quartic $C$ defined by an equation of this form satisfies $\Aut(C) \cong 2^2 : \mathfrak{S}_3 \cong \mathfrak{S}_4$.
\end{lemma}
\begin{proof}
We assume that $\Aut(C) \cong H : \mathfrak{S}_3$, where $H$ is non-trivial and acts diagonally, and $\mathfrak{S}_3$ is the group of permutation matrices. Thus, we may assume 
that $C$ is given by an equation as in Lemma \ref{lem:S3}. The polynomial $x_0x_1x_2^2 + x_1x_2x_0^2 + x_2x_0x_1^2$ is not semi-invariant under any diagonal automorphism, hence $d = 0$. By Lemma \ref{lem: detectfermat}, we have $c \neq 0$ and we can assume $c = 1$ by rescaling the coordinates. A diagonal automorphism for which $x_0^2x_1^2 + x_1^2x_2^2 + x_2^2x_0^2$ is semi-invariant has order $2$ and since $H$ is non-trivial and 
$C_3 \subseteq \Aut(C)$, we deduce that $H \cong 2^2$. But then the polynomial $x_0^3x_1 + x_1^3x_2 + x_2^3x_0 + x_0x_1^3 + x_1x_2^3 + x_2x_0^3$ is not $H$-semi-invariant, hence $b = 0$. This shows that 
$C$ is given by the stated equation and $a \not \in \mathbb{F}_3$ follows from the Jacobian criterion.

It remains to show that $\Aut(C) = 2^2 : \mathfrak{S}_3 \cong \mathfrak{S}_4$. By Lemma \ref{lem: detectfermat}, Corollary \ref{cor: FermatKlein}, and Theorem \ref{thm: possiblegroups}, it suffices to show that $\Aut(C)$ is not conjugate to a subgroup of $\Bbbk^2 : \GL_2(\Bbbk)$. If it were, then it would have to act through $\GL_2(\Bbbk)$, since $\frakS_4$ does not admit any non-trivial normal subgroup of order a power of $3$. In particular, the fixed locus of an element of order $3$ in $\frakS_4$ would be a line and this is not the case in our situation. Hence, $\Aut(C) \cong \frakS_4$.
\end{proof}

\subsection{Wild subgroups of $\Bbbk^2 : \GL_2(\Bbbk)$}
We already saw one conjugacy class of wild automorphisms in Lemma \ref{lem:3C}. The other conjugacy class is studied in the following lemma.

\begin{lemma} \label{lem:3B}
Let $C$ be a smooth plane quartic in characteristic $p = 3$. Then, $C$ is invariant under the automorphism $g: (x_0,x_1,x_2) \mapsto (x_0 + x_1,x_1,x_2)$ if and only 
if $C$ is given by an equation of the form
\beq\label{eqndp2C3}
    f_1(x_1,x_2)(x_0^3 - x_0x_1^2) + f_4(x_1,x_2) = 0,
\eeq
with $f_i$ homogeneous of degree $i$. The element $g$ belongs to the conjugacy class 
of type $3A_2$. 
\end{lemma}
\begin{proof}
The equation of $C$ follows from the structure of the invariant ring of $g$ (see the proof of Lemma \ref{lem:wildcyclic}). The closed subscheme of fixed points of $g$
 is $V(x_1)$. Our automorphism leaves the bitangent line given by the equation $f_1(x_1,x_2) = 0$ invariant. Since $g$ is of order 
 $3$, it also leaves each irreducible component of the preimage of $V(f_1)$ in $X$ invariant.  Thus, blowing down one of 
 the components of this preimage, we obtain an automorphism of order $3$ of a cubic surface that fixes a plane section
 (the preimage of $V(x_1)$) pointwise. It follows from \cite[Theorem 10.4]{DD} that this automorphism belongs to the conjugacy class 
$3A_2$ in $W(\sfE_6)$. Thus, the corresponding automorphism of $X$ is in the conjugacy class $3A_2$ in $W(\sfE_7)$.
 \end{proof}

\begin{lemma} \label{lem: Heisenberg}
Let $G$ be a wild subgroup of $\Bbbk^2 : \GL_2(\Bbbk)$.
If $9 \mid |G|$ and $G$ preserves a smooth plane quartic $C$, then $C$ is given by an equation of the form
\beq\label{eqndp2U9}
x_2(x_0^3 - x_0x_2^2) + x_1^4 + cx_1^2x_2^2
\eeq
for some $c \in \Bbbk$. Any plane quartic curve $C$ with such an equation is smooth. If $c = 0$, then $\Aut(C) \cong \PSU_3(9)$, and if $c \neq 0$, 
then $\Aut(C) \cong \mathcal{H}_3(3) : 2$.
\end{lemma}
\begin{proof}
Since $G$ contains no elements of order $9$ \cite[Theorem 8]{Dolgachev1}, it must contain two commuting elements $g_1,g_2$ of order $3$ 
that generate a subgroup $K$ of order $9$. 
Since $K$ preserves a bitangent, it acts faithfully on a cubic surface, hence, by \cite{DD}, 
it embeds into the Heisenberg group $\mathcal{H}_3(3)$. Since $K$ is abelian of order $9$ and 
$\mathcal{H}_3(3)$ is a non-split extension of $3^2$ by its center, the group $K$ contains 
the center of $\mathcal{H}_3(3)$, hence it contains elements of conjugacy class $3A_2$ in $W(\sfE_6)$. 
By \cite{DD}, there are at most two elements of 
conjugacy class $3A_2$ in an automorphism group of a cubic surface, so the other elements of $K$ are of 
conjugacy class $2A_2$. 
Back on $C$, this means that $K$ contains automorphisms of both conjugacy classes $2A_2$ and $3A_2$.

Without loss of generality, we may assume that the $g_i$ are upper triangular matrices, 
that $g_1$ is of class $2A_2$, and that $g_2$ is of class $3A_2$. Up to conjugation by an upper triangular matrix, 
we may assume that $g_1$ acts as $x_0 \mapsto x_0 + x_1$, as $x_1 \mapsto x_1 + x_2$, or as $x_0 \mapsto x_0 + x_2$. In the first two cases, the centralizer of $g_1$ in the group of 
strict upper triangular matrices consists of automorphisms of class $3A_2$, 
hence we are in the third case. Analogously to Lemma \ref{lem:3B}, we deduce 
that $C$ is given by an equation of the form
$$
f_1(x_1,x_2)(x_0^3 - x_0x_2^2) + f_4(x_1,x_2) = 0.
$$
After conjugating by a suitable upper triangular matrix, we may assume that the automorphism 
$g_2$ acts as $(x_0,x_1,x_2) \mapsto (x_0 + x_1,x_1 + \mu x_2,x_2)$ for some $\mu \neq 0$. 
Since $C$ is invariant under $g_2$, it is given by an equation of the form
$$
x_2(x_0^3 - x_0x_2^2) + (a x_1^4 + bx_1^3x_2 + c x_1^2x_2^2 + dx_1x_2^3 + e x_2^4) = 0
$$
with $a = - \frac{1}{\mu}$, $\mu^2 + c \mu + 1 = 0$, and $\mu^4a + \mu^3b + \mu^2c +\mu d = 0$. 

Now, we rescale $x_1$ so that $d = -b$, use a substitution of the form $x_0 \mapsto x_0 + \alpha x_1 + \beta x_2$ to set $b = d = e = 0$, and then scale $a$ to $1$. These substitutions commute with $g_1$, 
so $C$ is still invariant under $g_1$, but they change the shape of $g_2$. However, the equation is now in the simpler form
$$
x_2(x_0^3 - x_0x_2^2) + x_1^4 + cx_1^2x_2^2 = 0.
$$
Now, this equation is invariant under the substitutions 
$(x_0,x_1,x_2) \mapsto (x_0 + \lambda x_1 + \mu x_2, x_1 - \lambda^3 x_2,x_2)$ 
with $\lambda^9 - \lambda^3c + \lambda = 0$ and $\mu^3 - \mu + \lambda^{12} +c\lambda^6 = 0$.
These substitutions form a group isomorphic to $\mathcal{H}_3(3)$.
We also see that $C$ is invariant under the involution 
$\tau:x_1 \mapsto -x_1$, hence $\mathcal{H}_3(3) : 2 \subseteq \Aut(C)$.

It is known that a subgroup of $\Sp_6(2) = W(\sfE_7)/\langle \pm {\rm id} \rangle$ isomorphic to $\mathcal{H}_3(3):2$ is conjugate to a subgroup 
of $\mathcal{H}_3(3) : 8$ which is realized as a maximal subgroup of $\PSU_3(9)$ \cite[p.14]{ATLAS}. So, if $\Aut(C) \neq \PSU_3(9)$, then $\mathcal{H}_3(3)$ is normal in $\Aut(C)$ and $\Aut(C)$ contains an element $\tau'$ whose square is $\tau$. In particular, $\tau'$ acts on the fixed locus $V(x_2)$ of $g_1$ and this action is non-trivial, since $\tau$ acts non-trivially on $V(x_2)$. Since $\Aut(C)$ fixes the point $V(x_2) \cap C = [1,0,0]$, Theorem \ref{suz} shows that $\tau'$ generates a subgroup of $\Bbbk^{\times}$ of order $4$ which normalizes $\mathcal{H}_3(3)/\langle g_1 \rangle \cong 3^2$. Since $3^2$ acts on $V(x_2)$ as $x_0 \mapsto x_0 + \lambda x_1$ with $\lambda^9 - \lambda^3 c + \lambda$, the automorphism $\tau'$ cannot exist if $c \neq 0$. Hence, $\Aut(C) \cong \mathcal{H}_3(3) : 2$.
\end{proof}

Hence, in the following, we may assume that $|H| = 3$. There are three conjugacy classes of elements of order $3$ in $\Bbbk^2 : \GL_2(\Bbbk)$, namely
\begin{eqnarray*}
g_1: & (x_0,x_1,x_2) &\mapsto (x_0 + x_1,x_1,x_2), \\
g_2: & (x_0,x_1,x_2) &\mapsto (x_0,x_1 + x_2,x_2), \text{ and} \\
g_3: & (x_0,x_1,x_2) &\mapsto (x_0+x_1,x_1+x_2,x_2).
\end{eqnarray*}
In each case, the group $G$ is contained in the normalizer $N_i$ of $g_i$ in $\Bbbk^2 : \GL_2(\Bbbk)$. We have
\begin{eqnarray*}
N_1 &=& 
\left\{
\left( 
\begin{matrix}
1 & \alpha & \beta \\
0 & \gamma & 0 \\
0 & \delta & \varepsilon
\end{matrix}
\right) \mid \alpha,\beta,\delta \in \Bbbk, \varepsilon \in \Bbbk^{\times}, \gamma^2 = 1
\right\}, \\
N_2 &= &
\left\{
\left( 
\begin{matrix}
1 & 0 & \alpha \\
0 & \beta & \gamma \\
0 & 0 & \delta
\end{matrix}
\right) \mid \alpha,\gamma \in \Bbbk, \delta \in \Bbbk^{\times}, \beta^2 = 1
\right\}, \text{ and} \\
N_3 &= &
\left\{
\left( 
\begin{matrix}
1 & \alpha & \beta \\
0 & \gamma & \alpha \gamma + \gamma - 1 \\
0 & 0 & 1
\end{matrix}
\right) \mid \alpha,\beta \in \Bbbk, \gamma^2 = 1
\right\}. \\
\end{eqnarray*}
In each case, the centralizer of $g_i$ has index $2$ and is obtained by setting $\gamma = 1, \beta = 1,$ or $\gamma = 1$, respectively.
The element $g_2$ is conjugate in $\PGL_3(\Bbbk)$ to $g_1$ via a cyclic permutation of coordinates and this conjugation maps $N_2$ into $N_1$, so we only have to study the cases where $G$ contains $g_1$ or $g_3$.

If $G$ leaves a smooth quartic invariant, then $|G| = 2^n \cdot 3$ for some $n \geq 0$. We settled the cases where $n = 0$ in Lemma \ref{lem:3B} and Lemma \ref{lem:3C}. If $n \geq 1$, then $G$ contains an involution. Next, we will study involutions in $N_1$ and $N_3$.

Assume that $G$ contains $g_1$ and leaves invariant a smooth quartic $C$. 
Every involution in $N_1$ is conjugate to one where $\alpha = \beta = 0$. Then, it lies in the group of lower triangular matrices in $\GL_2(\Bbbk)$, and thus, after conjugating by a lower triangular matrix, we may assume that it acts diagonally. Thus, we have the three representatives
\begin{eqnarray*}
\tau_1: & (x_0,x_1,x_2) &\mapsto (x_0,-x_1,x_2), \\
\tau_2: & (x_0,x_1,x_2) &\mapsto (x_0,x_1,-x_2), \text{ and} \\
\tau_3: & (x_0,x_1,x_2) &\mapsto (x_0,-x_1,-x_2).
\end{eqnarray*}
In $N_1$, the elements $\tau_1,\tau_2,\tau_3$ represent the conjugacy class of involutions with parameters $(\gamma,\varepsilon) = (-1,1), (1,-1)$, and $(-1,-1)$, respectively.

\begin{lemma} \label{lem: g1}
Let $C$ be a smooth plane quartic in characteristic $p = 3$. Assume that $C$ is invariant under $G \subseteq \Bbbk : \GL_2(\Bbbk)$, $G$ contains $g_1$ and $9 \nmid |G|$. If $\Aut(C)$ contains an involution, then $C$ is one of the quartics of Lemma \ref{lem: Heisenberg}.
\end{lemma}
\begin{proof}
In Lemma \ref{lem:3B}, we have determined the $g_1$-invariant quartics. Determining the equations for the quartics which are additionally $\tau_i$-invariant is straightforward. Now, it remains to show that every $\tau_i$-invariant $C$ occurs in the family of Lemma \ref{lem: Heisenberg}.

Assume $G$ contains $\tau_1$. Then, either the equation of $C$ is $\tau_1$-anti-invariant and then $C$ contains the line $V(x_1)$, or $C$ is given by an equation of the form
 $$
    ax_2(x_0^3 - x_0x_1^2) + bx_1^4 + cx_1^2x_2^2 + dx_2^4 = 0.
    $$
    This curve is singular at the point $[\alpha,0,1]$ with $a\alpha^3 + d = 0$.

Assume $G$ contains $\tau_2$. Then, either $C$ is reducible or given by an equation of the form
$$
    ax_1(x_0^3 - x_0x_1^2) + bx_1^4 + cx_1^2x_2^2 + dx_2^4 = 0.
    $$
Rescaling coordinates, we may assume $a = 1$ and then we can use a substitution of the form $x_0 \mapsto x_0 + \alpha x_1$ to set $b = 0$. After rescaling $x_2$ and swapping $x_1$ and $x_2$, we obtain the normal form of Lemma \ref{lem: Heisenberg}.

Finally, assume that $G$ contains $\tau_3$. Then, $f_1(x_1,x_2)(x_0^3-x_0x_1^2)$ is $\tau_3$-anti-invariant while $f_4(x_1,x_2)$ is $\tau_3$-invariant, so either $f_1 = 0$ or $f_4 = 0$. In both cases, $C$ is singular.
\end{proof}

Finally, assume that $G$ contains $g_3$ and $9 \nmid |G|$. The unique conjugacy class of involutions in $N_3$ is 
represented by $\tau: (x_0,x_1,x_2) \mapsto (x_0,-x_1+x_2,x_2)$. All elements of $N_3$ which are not in this conjugacy class have order $1$ or $3$, hence $G \subseteq \mathfrak{S}_3$ with equality if and only if $G$ is conjugate to $\langle g_3,\tau \rangle$, which, in turn, is conjugate in $\GL_3(\Bbbk)$ to the group of permutation matrices. Hence, this case is reduced to Lemma \ref{lem:S3}.

\subsection{Tame automorphism groups of smooth plane quartics} \label{sec: tamedeg2}
 As we mentioned in the introduction, the pair $(X,\Aut(X))$ can be lifted to characteristic zero if $\Aut(X)$ is tame. Since $X$ does not admit global vector fields, specialization of automorphisms is injective, hence the lift $X'$ of $X$ to characteristic $0$ has the same automorphism group as $X$. 
 
In particular, this means that the list of tame groups that can be realized as automorphism groups of del Pezzo surfaces of degree $2$ in positive characteristic is contained in the list of \cite[Theorem 6.5.2]{CAG}. To finish the classification of tame groups, we have to show that the equations in this list have smooth reduction modulo $p$ and that the automorphism group of the reduction cannot be larger for a general choice of parameters. In the following, we explain how to do this for the first five groups in the list of \cite[Theorem 6.5.2]{CAG}.

\begin{itemize}
    \item ${\rm L}_2(7)$: This group is tame if $p \neq 3,7$. It is realized by the Klein quartic in every such characteristic.
    \item $4^2 : \frakS_3$: This group is tame if $p \neq 3$. Its order is $96$, so by Theorem \ref{thm: possiblegroups}, it is maximal among all groups that can occur as automorphism groups of smooth plane quartics in these characteristics. It is realized by the Fermat quartic.
    \item $4.\frakA_4$: This group is tame if $p \neq 3$. Since it contains an element of order $12$, it does not embed into $4^2: \frakS_3$ or ${\rm L}_2(7)$, hence it is maximal. The equation given in \cite[Table 6.1]{CAG} is smooth if $p \neq 3$, so this group is realized.
    \item $\frakS_4$: This group is tame if $p \neq 3$. The equation given in \cite[Table 6.1]{CAG} defines a $1$-dimensional family of smooth quartics with an action of this group in any such characteristic. Since all smooth quartics with larger automorphism group are unique, a generic choice of parameters will yield a smooth quartic with automorphism group $\frakS_4$.
    \item $4.2^2$: This group is tame. Again, the equation given in \cite[Table 6.1]{CAG} defines a $1$-dimensional family of smooth plane quartics with an action of this group in any characteristic. By the above and Table \ref{tab: classificationl}, all smooth quartics whose automorphism group contains $4.2^2$ are isolated, hence a generic member of this $1$-dimensional family has $\Aut(C) = 4.2^2$.
\end{itemize}

The other cases work similarly and are left to the reader. We conclude that all groups in the list of \cite[Theorem 6.5.2]{CAG} occur as the full automorphism group of some smooth plane quartic in characteristic $p$ if $p$ does not divide the order of the group.

\subsection{Wild automorphism groups of smooth plane quartics}
Here, we summarize the classification of wild groups of automorphisms of smooth plane quartics by collecting the results of the previous sections. We use the notation of \cite[Table 6.1]{CAG} for wild groups in characteristic $p$ that also occur in characteristic $0$. After a suitable change of coordinates, Equation \eqref{eqndp2U9} is the reduction modulo $3$ of the family of Type VIII in \cite[Table 6.1]{CAG}, hence we call it Type VIII. The Klein quartic in characteristic $3$ is the reduction modulo $3$ of the three quartics of type I, II, and III in \cite[Table 6.1]{CAG}.
Now, our classification of wild groups can be summarized as follows:

\begin{table}[h!]
    \centering
    \begin{tabular}{|c|c|c|c|c|c|} \hline
     Type&$\Aut(C)$  & Order &Equation& \# Parameters & Conditions\\ \hline \hline 
     I/II/III &$\PSU_3(9)$ & $6048$ & \eqref{eqndp2PSU39}& $0$ & -- \\ \hline
     IV &$\mathfrak{S}_4$ & $24$ &\eqref{eqndp2S4}& $1$ & $a \not \in \mathbb{F}_3$ \\ \hline
     VIII &$\mathcal{H}_3(3) : C_2$ &  $54$ &\eqref{eqndp2U9}& $1$ & $c\ne 0$ \\ \hline
     IX &$\mathfrak{S}_3$ & $6$ & \eqref{eqndp2S3}&$2$ & general \\ \hline
     XI &$C_3$ & $3$ & \eqref{eqndp2C3} & $2$ & general \\ \hline
    \end{tabular}
    \caption{Wild automorphism groups of smooth plane quartics in \mbox{characteristic $3$}}
    \label{tab: classificationl}
\end{table}

\subsection{Conjugacy classes}
As explained in Section \ref{S:1}, the conjugacy class of a tame automorphism $g$ of a del Pezzo surface $X$ can be determined using the Lefschetz fixed point formula. Let $\overline{g}$ be the composition of $g$ with the Geiser involution. Then, by \cite[Lemma 6.5.1]{CAG}, the pair $(e(X^g),e(X^{\bar{g}}))$ depends only on the Jordan form of the automorphism of $\bbP^2$ induced by $g$ and $\overline{g}$, hence so does the pair of conjugacy classes of $g$ and $\overline{g}$. By Lemma \ref{lem:3C} and Lemma \ref{lem:3B}, the pair of conjugacy classes of $g$ and $\overline{g}$ is also uniquely determined by the Jordan form in $\bbP^2$ if $g$ has order $3$, and it is not hard to extend this to the case where $g$ is wild of higher order. We summarize the translation between Jordan forms and pairs of conjugacy classes in Carter notation in Table \ref{table: carter1} in the Appendix. 
In that table, we give the orders of $g$ and $\bar{g}$ in the first column, the Jordan form of the induced automorphism of $\mathbb{P}^2$ in the second column, the conjugacy classes of $g/\bar{g}$ in the third column, and the traces of the actions of $g$ and $\bar{g}$ on $\sfE_7$ in the last column.

Using this dictionary, one can determine the conjugacy classes in $W(\sfE_7)$ of all automorphisms of del Pezzo surfaces of degree $2$. The result is summarized in Table \ref{tbl:autodp2} in the Appendix. In that table, we give the name of the relevant family of del Pezzo surfaces in the first column, following \cite[Table 6.1]{CAG}. In the second column, we note the characteristics in which the family occurs. The third and fourth columns give the group $\Aut(X)$ and its order. The remaining columns give the number of elements of a given Carter conjugacy class in $\Aut(X)$.

\section{Del Pezzo surfaces of degree 1}
In this section, we classify automorphism groups of del Pezzo surfaces $X$ of degree $1$. Recall that $X$ is a double cover of the quadratic cone $\mathbb{P}(1,1,2) \subseteq \mathbb{P}^3$ branched over a sextic curve $C$. Moreover, we have $\Aut(X) \cong 2.\Aut(C)$, but, in contrast to the case of degree $2$, this central extension is not necessarily split, so the classification of automorphism groups of del Pezzo surfaces of degree $1$ is more complicated than the classification of automorphism groups of possible branch curves $C$.

\subsection{The elliptic pencil}
As we have already observed in Section \ref{S:1}, the blow up of the unique base point $\frako$ of a del 
Pezzo surface $X$ of degree 
$1$ has the structure of an elliptic surface $\pi: Y \to \mathbb{P}^1$ with Weierstrass equation
\begin{equation*}
y^2+x^3+a_2(t_0,t_1)x^2+a_4(t_0,t_1)x+a_6(t_0,t_1) = 0.
\end{equation*}
We have 
$$\pi^*\calO_{\bbP^1}(1) \cong \calO_Y(-K_Y).$$
Since $-K_X$ is ample, the fibration has no reducible fibers.

If $p = 3$, and the absolute invariant $j(F_\eta)$ of the generic fiber $F_{\eta}$ of $\pi$ (considered as an elliptic curve 
over the field $\Bbbk(\mathbb{P}^1)$ of rational functions on the base of the fibration) is equal to zero, we may assume that $a_2 = 0$. 
In this case $a_4\ne 0$ (otherwise $F_\eta$ is singular) and $a_6\ne 0$ (otherwise the surface $X$ has singular 
points over $V(a_4)$).

Obviously, any $g\in \Aut(X)$ fixes the point $\frako$ and hence lifts to an automorphism $\tilde{g}$ of $Y$ that preserves 
the corresponding section $S$. Conversely, any automorphism of $Y$ leaving invariant the section $S$ descends to $X$.

Thus
$$\Aut(X) \cong \Aut_S(Y) \coloneqq \{g\in \Aut(Y): g(S) = S\}.$$
Let
\beq
\phi:\Aut(X)\to \GL(T_{\frako}(X))
\eeq 
be the natural representation in the tangent space $T_{\frako}(X)$ of $X$ at the point $\frako$.  
Since any tame automorphism acts faithfully on 
$T_{\frako}(X)$, the kernel 
$$H_0: = \Ker(\phi)$$ 
of $\phi$ is a normal subgroup of a $p$-Sylow subgroup of $\Aut(X)$.
The projectivization of the representation $\phi$ defines a homomorphism
\beq
\bar{\phi}:\Aut_S(Y) \to \Aut(S) \cong \Aut(\bbP^1).
\eeq
We may view $\bar{\phi}$ as the natural action of $\Aut(X)$ on the base of the fibration $\pi$.

The action $\bar{\phi}$ allows us to write $\Aut(X)$ as an extension $H.P$, where
\begin{eqnarray*}
H & \coloneqq & \Ker(\bar{\phi}) \text{ and }\\
P & \coloneqq &\bar{\phi}(\Aut(X)) \subseteq \Aut(\mathbb{P}^1).
\end{eqnarray*}
In the following, we collect some preliminary restrictions on the groups $H$ and $P$.

The group $P$ is a subgroup of $\PGL_2(\Bbbk)$, so we can use the classification of finite subgroups of $\PGL_2(\Bbbk)$ to study it. The classification of finite tame subgroups of 
$\PGL_2(\Bbbk)$ 
coincides with the well-known classification of finite subgroups of $\PGL_2(\mathbb{C})$: The finite subgroups of $\PGL_2(\mathbb{C})$ are exactly the polyhedral groups, that is, the groups $C_n,D_n,\frakA_4,\frakS_4,$ and $\frakA_5$.
The classification of 
wild subgroups of $\PGL_2(\Bbbk)$ is contained in Theorem \ref{suz}.

The group $H$ is an extension $H_0: C_n$, where $C_n$ is a cyclic subgroup of order $n$ prime to $p$. Since $p$ is odd, $C_n$ contains the Bertini involution, so $2 \mid n$.
We can also describe $H$ as the automorphism group of the generic fiber 
$F_\eta$ of the elliptic fibration. The structure of this group is well-known (see e.g. \cite[Appendix A]{Silverman}) and we have the following normal forms:

\begin{lemma} \label{lem:autgenericfiber}
Let $X$ be a del Pezzo surface of degree $1$, let $F_{\eta}$ be the generic fiber of the anti-canonical pencil $|-K_X|$ with $j$-invariant $j \coloneqq j(F_{\eta})$, and let $H$ and $H_0$ be as above.
\begin{itemize}
    \item[(i)] If $j \not\in \{0,1728\}$, then $H_0 = \{1\}$ and $H = C_2$.
    \item[(ii)] If $p \neq 3$ and $j = 1728$, then $H_0 = \{1\}$, $H = C_4$ and $X$ admits the equation
    $$
    y^2 + x^3 + a_4x = 0.
    $$
    \item[(iii)] If $p \neq 3$ and $j = 0$, then $H_0 = \{1\}$, $H = C_6$ and $X$ admits the equation
    $$
    y^2 + x^3 + a_6 = 0.
    $$
    \item[(iv)] If $p = 3$ and $j = 0$, then $H_0 \subseteq C_3$ and $H$ is $C_2$ or $C_6$. Moreover, $H_0 = C_3$ if and only if $a_4$ is a square.
\end{itemize}
\end{lemma}
\begin{proof}
The first three cases are well-known, and $H$ acts as $(x,y,t_0,t_1)\mapsto (u^2x,u^3y,t_0,t_1)$, where $u^n = 1$ with $ n = 2,4,6$, respectively.

In Case (iv), the automorphism group of the geometric generic fiber of $\pi$ is the alternating group $\frakA_4$. To show that $H_0 \in \{1,C_3\}$ and $H/H_0 = C_2$, it thus suffices to show that no automorphism of order $4$ can descend to $F_{\eta}$.
Assume that $F_{\eta}$ admits an automorphism of order $4$. Its fixed locus on $X$ is a $2$-torsion section of the elliptic pencil $\pi: Y \to \mathbb{P}^1$, which, after a suitable change of the $x$-coordinate, we may assume to be given by $x = y = 0$. In other words, $X$ admits a Weierstrass equation of the form
$$
y^2 + x^3 + a_4x = 0.
$$
This equation is singular over the roots of $a_4$, so $X$ is not a del Pezzo surface.

By \cite[Appendix A, Proposition 1.2]{Silverman} the $\Bbbk(\mathbb{P}^1)$-linear substitutions that preserve equation \eqref{eq:dp1} are of the form
$$
(x,y,t_0,t_1) \mapsto (u^2x+r_2,u^3y,t_0,t_1),
$$
where $u^4 = 1$ and 
$$
r_2^3+r_2a_4+(1-u^2)a_6 = 0.
$$
Since $H/H_0 = C_2$, every element of $H$ satisfies $u^2 = 1$, so $r_2^3 + a_4r_2 = 0$. A non-zero $r_2$ solving this equation exists if and only if $a_4$ is a square. Hence, $H_0 = C_3$ if and only if $a_4$ is a square.
\end{proof}

\subsection{List of possible groups} The following lemma shows that wild automorphism groups of del Pezzo surfaces of degree $1$ in 
odd characteristic can only exist if $p = 3,5$. 

\begin{lemma}\label{lem:wildcyclic1} 
Assume $C_p$ acts faithfully on a del Pezzo surface $X$ of degree $1$. Then, $p = 3$ or $p = 5$. 
\end{lemma}

\begin{proof} By Theorem \ref{thm:weylgroups2}, the prime divisors of the order of $W(\sfE_8)$ are $2,3,5$, and $7$. So, 
it suffices to exclude $p = 7$. An automorphism $g$ of order $7$ acts on the set of $120$ tritangent 
planes preserving one of them. It fixes each of the $(-1)$-curves in the preimage pair.
Blowing one of them down, we obtain that $g$ descends to a wild automorphism of a del Pezzo surface of degree $2$.
 However, by Corollary \ref{cor: FermatKlein}, we know that there are no wild automorphisms of order $7$ on del Pezzo surfaces of degree $2$.
\end{proof}

As in the case of degree $2$, the restrictions obtained so far allow us to give a preliminary list of possible automorphism groups of del Pezzo surfaces of degree $1$.

\begin{corollary} \label{cor: possiblegroups}
Let $G$ be a finite group acting faithfully on a del Pezzo surface $X$ of degree $1$. Then, one of 
the following holds.
\begin{enumerate}
    \item $G$ is tame and $G \cong H : P$, where $H$ is cyclic and $P$ is a polyhedral group.
    \item $G$ is wild, $p = 5$, and $G \cong H : P$, where $H$ is tame and cyclic, and $$P \in \{ 5^n:C_m, {\rm L}_2(5^n), \PGL_2(5^n) \}$$ for some $n \geq 1$ and $(5,m) = 1$.
    \item $G$ is wild, $p = 3$, and $G \cong H : P$, where $H$ is cyclic, and 
    $$P \in \{ 3^n:C_m, D_m, \frakA_5, {\rm L}_2(3^n), \PGL_2(3^n) \}$$ for some $n \geq 0$ and $(3,m) = 1$.
\end{enumerate}
\end{corollary}
\begin{proof}
By Lemma \ref{lem:autgenericfiber}, we have $G \cong H : P$, where $H$ is a cyclic group and $P \subseteq \PGL_2(\Bbbk)$.

If $G$ is tame, then it lifts to characteristic $0$, so $P$ embeds into $\PGL_2(\mathbb{C})$, hence it is a polyhedral group. 
If $G$ is wild, then $p = 3,5$ by Lemma \ref{lem:wildcyclic1}. The description of the possible $P$ follows from Theorem \ref{suz}. 
\end{proof}

In the following sections, we first classify the wild automorphism groups of del Pezzo surfaces of degree $1$ in characteristic $5$ and $3$. Afterwards, we study the tame automorphism groups using the approach of Section \ref{sec: tamedeg2} in Section \ref{sec: tamedeg1}.

\subsection{Wild automorphism groups in characteristic $5$} \label{sec: char5}
In this section, we classify wild automorphism groups of del Pezzo surfaces of degree $1$ in characteristic $p = 5$.

 \begin{theorem}\label{p=5} Let $X$ be a del Pezzo surface of degree $1$ in characteristic $p = 5$. Assume that $G = \Aut(X)$ is wild. Then, one of the following cases occurs.
 \begin{itemize}
 \item[(i)] $G \cong 2.D_{10}$ and $X$ is given by
 \beq\label{eqndp1D10}
y^2+x^3+ct_0^4x+t_0t_1(t_1^4-t_0^4) = 0 \text{ with } c \neq 0.
 \eeq
 \item[(ii)] $G\cong C_6.\PGL_2(5) \cong 3\times (2.\PSL_2(5).2) \cong 3\times (\SU_2(25).2)$ 
  \beq\label{eqndp1660}	
  y^2+x^3+t_0t_1(t_1^4-t_0^4) = 0
  \eeq
 \end{itemize}
 \end{theorem}

\begin{proof}
Let $g$ be a wild automorphism of order $5$. Since $g\not\in H = \Ker(\bar{\phi})$ by Lemma \ref{lem:autgenericfiber}, it 
acts non-trivially on the base 
of the elliptic fibration. As $g$ has one 
fixed point in $\bbP^1$, we may assume this point to be $[0,1]$, and that $g$ acts by $(t_0,t_1)\mapsto (t_0,t_0+t_1)$.
The roots of $a_4$ coincide with the roots of the $j$-invariant, hence they are preserved by $g$. But $a_4$ has degree $4$, so $a_4 = c_1t_0^4$ for some $c_1 \in \Bbbk$.
Then, the discriminant of the elliptic fibration is given by $\Delta = 4a_4^3 + 27a_6^2$, so $a_6$ is preserved by $g$ as well and $t_0^2 \nmid a_6$, for otherwise $\Delta$ would vanish with multiplicity at least $3$ at $[0,1]$ and the fiber over this point would be reducible. Thus, choosing coordinates such that $a_6$ has a root at $[1,0]$, we may assume that 
$$a_4 = ct_0^4,\quad a_6 = t_0\prod_{i=0}^4(t_1+ it_0) = t_0t_1(t_1^4 - t_0^4).$$
In other words, the surface $X$ admits the equation 
\beq
y^2+x^3 + ct_0^4x+t_0t_1(t_1^4-t_0^4) =  0.
\eeq

If $c \ne  0$, then $j(F_{\eta}) \not \in \{0,1728\}$, hence $H \cong C_2$. The group $P = \bar{\phi}(\Aut(X)) \subseteq \Aut(\mathbb{P}^1)$ preserves $V(a_4) = [0,1]$ and the set $V(a_6) = \mathbb{P}^1(\mathbb{F}_5)$, hence $P \subseteq C_5 : C_4$. Since $c \neq 0$, the discriminant $\Delta$ has a double root at $[0,1]$ and $10$ simple roots.
Now, $P$ cannot be $C_5 : C_4$, since the square of an automorphism of order $4$ would fix $V(a_4)$ and at least two simple roots of $\Delta$, hence all of $\mathbb{P}^1$. However, $P$ contains an involution lifting to $X$ as an automorphism of order $4$ given by $(x,y,t_0,t_1) \mapsto (-x,\lambda y,t_0,-t_1)$ with $\lambda^4 = 1$. Hence, in this case, $\Aut(X)$ is the binary dihedral group of order $20$.

If $c = 0$, then $j(F_{\eta}) = 0$, hence $H \cong C_6$. The group $P = \bar{\phi}(\Aut(X)) \subseteq \Aut(\mathbb{P}^1)$ preserves $V(a_6) = \mathbb{P}^1(\mathbb{F}_5)$, hence $P \subseteq \PGL_2(\mathbb{F}_5)$. Conversely, every element of $\GL_2(\mathbb{F}_5)$ sends $a_6$ to $\lambda^6 a_6$ for some $\lambda \in \Bbbk^{\times}$, and composing this substitution with $(x,y)$ to $(\lambda^2 x, \pm \lambda^3 y)$ gives a lift of every element of $\PGL_2(\mathbb{F}_5)$ to $\Aut(X)$.
\end{proof}

\subsection{Wild automorphism groups in characteristic $3$}
The case $p = 3$ is more difficult. If $\Aut(X)$ is a wild group, then so is $H$ or $P$.
Recall that $X$ is given by an equation of the form
$$
y^2 + x^3 + a_2x^2 + a_4x + a_6 = 0.
$$
The discriminant $\Delta$ and $j$-invariant $j$ of the associated elliptic fibration $\pi$ are
\begin{eqnarray*}
\Delta &=& -a_2^3a_6 - a_2^2a_4^2 - a_4^3 \\
j &=& \frac{a_2^6}{\Delta}
\end{eqnarray*}
We treat the cases $j(F_\eta) \neq 0$ and $j(F_\eta) = 0$ separately.

\begin{theorem}
Let $X$ be a del Pezzo surface of degree $1$ in characteristic $p = 3$ such that $\Aut(X)$ is wild. Assume 
that $j(F_{\eta}) \neq 0$. Then, $X$ admits one of the following two Weierstrass equations:
\beq \label{eq:3.1-1}
y^2 + x^3 + at_0^2x^2 + t_0t_1(t_1^2-t_0^2)x + bt_0^6 + ct_0^4t_1^2 + ct_0^2t_1^4 + ct_1^6
\eeq
with $a,b,c \neq 0$ and $ac \neq 1$, or
\beq \label{eq:3.1-1b}
y^2 + x^3 + t_0^2x^2 + t_0t_1^3x + t_1^6 + at_0^3t_1^3 + bt_0^5t_1
\eeq 
with $b \neq 0$. Moreover,
\begin{enumerate}
    \item if $X$ is given by Equation \eqref{eq:3.1-1}, then $\Aut(X) = C_6$.
    \item if $X$ is given by Equation \eqref{eq:3.1-1b}, then $\Aut(X) = C_2 \times C_3^2$.
\end{enumerate}
\end{theorem}

\begin{proof}
Note that $j(F_{\eta}) \neq 0$ is equivalent to $a_2 \neq 0$. 
Moreover, $H = C_2$, so if $\Aut(X)$ is wild, then $P$ is wild. 
We may assume that $P$ contains the substitution $g: t_1 \mapsto t_1 + t_0$. 
Since, the roots of $a_2$ are the roots of the $j$-map, $a_2$ is $g$-invariant, 
hence, after rescaling $x,y$, we may assume that $a_2 = t_0^2$.  Note that, under this assumption on $a_2$, substitutions of the form $x \mapsto x + b_2$ preserve $a_2$ and the class $\overline{a_4}$ of $a_4$ modulo $t_0^2$.

Now, a suitable substitution in $x$ allows us to eliminate the monomials 
$t_0^4,t_0^3t_1,$ and $t_0^2t_1^2$ in $a_4$, so we can write $a_4 = at_0t_1^3 + bt_1^4$. As $\overline{a_4}$ is $g$-semi-invariant, we must have $b = 0$. After scaling $t_0$, we may assume 
that $a_4 = \epsilon t_0t_1^3$, where $\epsilon\in \{0,1\}$, and that $a_2 = at_0^2$ for some $a \neq 0$.

Since $g$ lifts to $X$ and $g$ has odd order, there is an element of order $3$ in $\Aut(X)$ that induces $g$. 
By \cite[Appendix A, Proposition 1.2]{Silverman}, this means that there exists $b_2 \in \Bbbk[t_0,t_1]_2$ such that the equation of $X$ is preserved by $(y,x,t_0,t_1) \mapsto (y,x+b_2,t_0,t_1+t_0)$. This $b_2$ satisfies
\begin{eqnarray}
a_4(t_0,t_1) &=& a_4(t_0,t_1+t_0) - a b_2 t_0^2 \label{eq: a_6jnon0part1} \\
a_6(t_0,t_1) &=& b_2^3 + at_0^2b_2^2 + a_4(t_0,t_1+t_0)b_2 + a_6(t_0,t_1+t_0). \label{eq: a_6jnon0}
\end{eqnarray}

If $\epsilon = 0$, Equation \eqref{eq: a_6jnon0part1} shows that $b_2 = 0$ and comparing the coefficients of $t_0^2t_1^4$ in Equation \eqref{eq: a_6jnon0} shows that $a_6$ does not contain the monomial $t_0t_1^5$. Computing the partial derivatives, we see that this implies that $X$ is singular over $[0,1]$.

Thus, we have $\epsilon = 1$. The substitution $x \mapsto x + a^{-1}t_0t_1$ transforms $a_4$ to $t_0t_1^3 - t_0t_1^3$, thereby making it $g$-invariant. With this choice of $a_4$, Equation \eqref{eq: a_6jnon0part1} yields $b_2 = 0$. We write $a_6 = bt_0^6 + ct_0^5t_1 + dt_0^4t_1^2 + et_0^3t_1^3 + ft_0^2t_1^4 + gt_0t_1^5 + ht_1^6$. 
Comparing coefficients of the monomials in Equation \eqref{eq: a_6jnon0} above, we find 
$$(e,f,g,h) = (-c,d,0,d).$$
We obtain the normal form
$$
 y^2 + x^3 + at_0^2x^2 + t_0t_1(t_1^2-t_0^2)x + bt_0^6 + c(t_0^5t_1 - t_0^3t_1^3) + d(t_0^4t_1^2 + t_0^2t_1^4 + t_1^6).
$$

If $ad \neq 1$, a substitution of the form $(t_1,x) \mapsto (t_1 + \lambda t_0, x + a^{-1}(\lambda^3-\lambda))$ can be used to set $c = 0$. This yields Equation \eqref{eq:3.1-1} after renaming $d$ to $c$. Computing the partial derivatives, we see that Equation \eqref{eq:3.1-1} defines a smooth surface if and only if $b,c \neq 0$ and $bc \neq 1$. Moreover, note that $a \neq 0$ is equivalent to $j(F_{\eta}) \neq 0$. The discriminant of this equation is 
$$
\Delta = t_0^3(-a^3bt_0^9 + (a^2 - a^3c)t_0^7t_1^2 + t_0^6t_1^3 + (a^2 - a^3c)t_0^5t_1^4 + (a^2 - a^3c)t_0^3t_1^6 + t_1^9).
$$
To compute $\Aut(X)$, note that the associated elliptic fibration $\pi$ has a unique cuspidal fiber over $[0,1]$, so every element of $P$ acts as $h: t_1 \mapsto A t_1 + B t_0$. Comparing what happens to the coefficients of $x^3, t_0x^2,$ and $t_0t_1^3x$ under such a substitution, we deduce that $A = 1$. In particular, $\Delta$ must be $h$-invariant. The coefficient of $t_0^6t_1^6$ in $h^* \Delta - \Delta$ is $(B-B^3)(a^2 - a^3c)$. Since $a^2 - a^3c \neq 0$, this shows that $B^3 = B$, hence $h$ is a multiple of $g$, so $\Aut(X) = C_6$.

If $ad = 1$, we can apply a substitution of the form $$(t_1,x) \mapsto (t_1 + \lambda t_0, x - a^{-1}t_0t_1 + a^{-1}(\lambda^3 - \lambda)t_0^2)$$ for a suitable $\lambda$ and then rescale $t_0$ and $t_1$ to obtain Equation \eqref{eq:3.1-1b}. Note that the rescaling of $t_0$ and $t_1$ does not necessarily preserve our description of $g$. To calculate $\Aut(X)$, we use that, by the same argument as in the previous paragraph, every element of $P$ acts as $h: t_1 \mapsto t_1 + Bt_0$. Analogously to Equation \eqref{eq: a_6jnon0}, the condition that an automorphism with $A = 1$ lifts to $X$ is that there exists $b_2 \in \Bbbk[t_0,t_1]_2$ such that
\begin{eqnarray*}
a_4(t_0,t_1) &=& a_4(t_0,t_1 + B t_0) - b_2t_0^2 \\
a_6(t_0,t_1) &=& b_2^3 + t_0^2b_2^2 + a_4(t_0,t_1+ B t_0)b_2 + a_6(t_0,t_1+ B t_0).
\end{eqnarray*}
The first equation yields $b_2 = B^3 t_0^2$. The only non-trivial condition in the second equation is for the coefficient of $t_0^6$, where we get
$$
0 = B^9 + aB^3 + bB.
$$
Since $b \neq 0$, this equation has exactly $9$ solutions, hence $P = C_3^2$ and $\Aut(X) = C_2 \times C_3^2$. 
\end{proof}

\begin{theorem}\label{classification}
Let $X$ be a del Pezzo surface of degree $1$ in characteristic $p = 3$ such that $\Aut(X)$ is wild. Assume that $j(F_{\eta}) = 0$. 
\begin{enumerate}
    \item If $H$ is tame, then $H= C_2$ and $X$ admits a Weierstrass equation of the form
    \beq\label{eq:3.2-1}
    y^2 + x^3 + t_0t_1(t_1^2-t_0^2)x + at_0^6 + bt_0^4t_1^2 + bt_0^2t_1^4 + bt_1^6 = 0
   \eeq
    with $a,b \neq 0$, $b + c \neq 0$. Moreover,
    \begin{enumerate}
        \item if $a \neq b$, then $P = C_3$ and $\Aut(X) = C_6$.
        \item if $a = b$, then $P = {\rm L}_2(3)$ and $\Aut(X) = \SL_2(3)$. 
    \end{enumerate}
    \item If $H$ is wild and $\pi$ has two singular fibers, then $H = C_6$ and $X$ admits a Weierstrass equation of the form
    \beq\label{eq:3.2-2}
 y^2 + x^3 - t_0^2t_1^2x + at_0^5t_1 + bt_0^4t_1^2 + ct_0^2t_1^4 + t_0t_1^5 = 0
    \eeq
    with $a \neq 0$. Moreover,
    \begin{enumerate}
        \item if $b,c \neq 0$, and $a^2c^4\ne b^4$, then $P = \{1\}$ and $\Aut(X) = C_6$.
        \item if $b,c\ne 0$ and $a^2c^4=b^4$, then $P = C_2$ and $\Aut(X) = C_2 \times C_6$ 
        \item if $b = c = 0$, then $P = C_4: C_2\cong D_8$ and $\Aut(X) = C_6 . D_8$.
     \end{enumerate}
    \item If $H$ is wild and $\pi$ has one singular fiber, then $H = C_6$ and $X$ admits a Weierstrass equation of the form
    \beq\label{eq:3.2-3}
 y^2 + x^3 - t_0^4x + at_0^5t_1 + bt_0^4t_1^2 + t_0t_1^5 = 0.
    \eeq
    Moreover,
    \begin{enumerate}
    \item if $b \neq 0$, then $P = \{1\}$ and $\Aut(X) = C_6$.
    \item if $b = 0$ and $a \neq 0$, then $P = C_2$ and $\Aut(X) = 2.D_6$.
    \item if $a = b = 0$, then $P = C_{10}$ and $\Aut(X) = C_6.C_{10}$. 
    \end{enumerate}
\end{enumerate}
\end{theorem}

\begin{proof}
Recall that $j(F_{\eta}) = 0$ implies that $X$ admits a Weierstrass equation of the form
$$
y^2 + x^3 + a_4x + a_6 = 0.
$$

\bigskip
\smallskip
\underline{Case 1}: $H$ is tame.
\smallskip

Since $\Aut(X)$ is wild and $H = C_2$ by Lemma \ref{lem:autgenericfiber}, the group $P$ is wild. 
We may assume that $P$ contains the substitution $g: t_1 \mapsto t_1 + t_0$. The ring of invariants 
$\Bbbk[t_0,t_1]^{(g)}$ is generated by $t_0$ and $t_1(t_1^2-t_0^2)$. 
Since $\Delta = -a_4^3$, $g$ preserves the set of roots of $a_4$ and since $g$ is wild, it has only one fixed point on $\bbP^1$. 
On the other hand, by Lemma \ref{lem:autgenericfiber}, $a_4$ is not a square. So, $a_4$ has 4 distinct roots, and we may assume 
that $[0,1]$ and $[1,0]$ are among them, so that
$a_4 = t_0t_1(t_1^2-t_0^2)$ and, after simplifying,
$$a_6 = at_0^6+bt_0^4t_1^2+ct_0^2t_1^4 + dt_1^6.$$ 
Since $g$ lifts to $X$, there exists $b_2 =  \alpha t_0^2 + \beta t_0t_1 + \gamma t_1^2\in \Bbbk[t_0,t_1]_2$ such that
$$
b_2^3 + a_4(t_0,t_0+t_1)b_2 + a_6(t_0,t_0+t_1) = a_6(t_0,t_1).
$$
Comparing the coefficients, we obtain
$$\alpha = \beta = \gamma = b+c+d = 0$$
and
$$
b = c = d.
$$
This yields the normal form in the theorem. Taking partial derivatives, we find that the Weierstrass surface $X$ is smooth if and only if 
$a,b \neq 0$.

It remains to determine the full group $P$. 
Since $P$ acts faithfully on the set of four roots of $\Delta = -a_4^3$, Theorem \ref{suz} 
shows that 
\beq\label{suzuki2}
P \in \{C_3, 3 : C_2, 3 : C_4,{\rm L}_2(3),\PGL_2(\mathbb{F}_3)\}.
\eeq

If $P \in \{3 : C_2, 3 : C_4, \PGL_2(\mathbb{F}_3)\}$, then $P$ contains an involution that normalizes the subgroup generated by $g$ and such that conjugation by this involution maps $g$ to $g^{-1}$. Without loss of generality, we may assume that this involution is $(t_0,t_1) \mapsto (t_0,-t_1)$. Then, an explicit computation shows that such an involution never lifts to $X$. 

If $P = {\rm L}_2(3)$, then $P$ contains the involution $\tau: (t_0,t_1) \mapsto (-t_1,t_0)$. Note that $a_4(-t_1,t_0) = a_4(t_0,t_1)$. This $\tau$ lifts to $X$ if and only if there exists $b_2 =  \alpha t_0^2 + \beta t_0t_1 + \gamma t_1^2\in \Bbbk[t_0,t_1]_2$ such that
$$
b_2^3 + a_4(t_0,t_1)b_2 + a_6(-t_1,t_0) = a_6(t_0,t_1).
$$
Comparing coefficients, we see that $b_2 = 0$, hence $\tau$ lifts to $X$ if and only if $a_6(-t_1,t_0) = a_6(t_0,t_1)$. This holds if and only if $a = b$. We also see that, under this assumption, the natural $\SL_2(3)$-action on $t_0$ and $t_1$ lifts to the surface $X$, and the center of $\SL_2(3)$ acts the Bertini involution.

\smallskip
\underline{Case 2}: $H$ is wild, $a_4$ has more than one root.
\smallskip

Since $H$ is wild, Lemma \ref{lem:autgenericfiber} shows that $a_4$ is a square. Since it has more than one root, we may assume that $a_4 = - t_0^2t_1^2$. A suitable substitution in $x$ allows us to describe $X$ by an equation of the following form:
$$
y^2 + x^3 - t_0^2t_1^2x + at_0^5t_1 + bt_0^4t_1^2 + ct_0^2t_1^4 + d t_0t_1^5 = 0.
$$
As in the previous case, computing the partial derivatives, we find that $a,d\ne 0$,
 so we can rescale $d$ to $1$.
By Lemma \ref{lem:autgenericfiber}, we have $H = C_6$. It remains to determine $P$.

After replacing elements of $g$ by scalar multiples, we may assume that $g^*(a_4) = a_4$ for every $g \in P$. Note that such a $g$ either acts diagonally or anti-diagonally, hence $g^*(a_6)$ does not contain the monomials $t_0^6,t_0^3t_1^3$, and $t_1^6$. The condition that $g$ can be lifted to an automorphism of $X$ is that there exists 
$b_2 = s_1t_0^2+s_2t_0t_1+s_3t_1^2$ such that
\beq\label{identity1}
b_2^3+ a_4b_2 +g^*(a_6) = \mp a_6.
\eeq
Comparing the coefficient of $t_0^6,t_0^3t_1^3,$ and $t_1^6$, we find that $b_2 = st_0t_1$ with $s(s^2-1) = 0$.

Now, assume first that $g$ acts as $(t_0,t_1) \mapsto (\alpha t_0,\alpha^{-1} t_1)$ with $\alpha^2 \neq 1$. Comparing the coefficients in Equation \eqref{identity1}, we obtain
$$
(\alpha^4 \pm 1)a = (\alpha^2 \pm 1)b = (\alpha^{-2} \pm 1) c = \alpha^{-4} \pm 1 = 0,
$$
where the choice of sign is compatible with Equation \eqref{identity1}. If $b \neq 0$ or $c \neq 0$, then $\alpha^2 = -1$ and the sign on the right-hand side of Equation \eqref{identity1} is a minus, so that $\alpha^{-4} + 1 = 0$ cannot hold. Hence, $g$ can exist only if $b = c = 0$.  In this case, the unique condition for the liftability of $g$ is $\alpha^8 = 1$, so the group of such $g$ determines a cyclic subgroup of order $4$ in $\PGL_2(\Bbbk)$. Moreover, note that if $g$ has order $4$, then every lift of $g$ to $X$ with $b_2 = 0$ has order $8$.

Next, assume that $g$ acts as $(t_0,t_1)\mapsto (\alpha t_1,\alpha^{-1} t_0)$. Equation \eqref{identity1} becomes
$$
(a \alpha^4 \pm 1) = (b \alpha^2 \pm c) = (c \alpha^{-2} \pm b) = (\alpha^{-4} \pm a ) = 0.
$$
Thus, $a^2 = \alpha^{-8}$. If $b \neq 0$, then $c \neq 0$. Moreover, these equations admit a solution if and only if $a^2c^4 = b^4$. Note that in this case the involution $g$ lifts to an involution of $X$. 

In summary, if $(b,c) \neq (0,0)$ and $a^2c^4 \neq b^4$, then $P$ is trivial, so $\Aut(X) = H = C_6$. 
If $(b,c) \neq (0,0)$ and $a^2c^4 = b^4$, then $P = C_2$ and $H = C_6$. It is obvious that a generator of 
$P$ commutes with a generator of $H$, so  $\Aut(X) = H \times P = C_6 \times C_2$.
If $(b,c) = (0,0)$, then $\Aut(X)$ is a non-split extension of $P \cong C_4 : C_2$ by $H \cong C_6$. 

\smallskip
\underline{Case 3}: $H$ is wild, $a_4$ has exactly one root.
\smallskip

Since $H$ is wild, Lemma \ref{lem:autgenericfiber} shows that $a_4$ is a square. Since it has exactly one root, we may assume that $a_4 = - t_0^4$. Since $X$ is smooth, we must have $t_0^2 \nmid a_6$. A suitable substitution in $x$ allows us to describe $X$ via an equation of the following form:
$$
y^2 + x^3  - t_0^4 x + at_0^5t_1 + bt_0^4t_1^2 + ct_0^2t_1^4 + t_0t_1^5.
$$
Taking partial derivatives, we see that $X$ is always smooth. Using a substitution of the form $t_1 \mapsto t_1 + \lambda t_0, x \mapsto x + \alpha t_0^2 + \beta t_0t_1$, we may assume additionally that $c = 0$.
By Lemma \ref{lem:autgenericfiber}, we have $H = C_6$. It remains to determine $P$.

Since $P$ preserves the unique root of $a_4$, we may assume that $P$ consists of transformations 
of the form $(t_0,t_1)\mapsto (t_0, A t_0 + B t_1)$. In order for this transformation to lift to $X$, there must exist
$b_2 = \alpha t_0^2+\beta t_0t_1+\gamma t_1^2$ such that
 \beq \label{identity2}
 b_2^3 + a_4b_2 + a_6(t_0,A t_0 + B t_1) = \pm a_6(t_0,t_1).
\eeq
Comparing the coefficients of $t_0^2t_1^4$, we see that $A = 0$. Comparing the coefficients of $t_0^6, t_0^3t_1^3$, and $t_1^6$, we find that $\beta = \gamma = 0$ and $\alpha^3 - \alpha = 0$.
Then, Condition \eqref{identity2} becomes
$$
a(B \mp 1) = b(B^2 \mp 1) = B^5 \mp 1 = 0.
$$

If $B = 1$, then the transformation is the identity. If $B = -1$, then $b = 0$ and, conversely, if $b = 0$, the transformation with $B = -1$ lifts to $X$ (as an automorphism of order $4$).
If $B$ is a primitive $5$-th root of unity, then $a = b = 0$ and again, the transformation lifts to $X$ under these conditions.
\end{proof}

\subsection{Tame automorphism groups} \label{sec: tamedeg1}
We showed in Section \ref{sec: tamedeg2} that all automorphism groups of del Pezzo surfaces 
of degree $2$ in characteristic $0$ appear as the tame automorphism groups of 
del Pezzo surfaces of degree $2$ in positive characteristic.  It turns out that this is not true anymore for del Pezzo surfaces 
of degree $1$. If $p = 3,5$, some of these tame groups appear only as proper subgroups of the full groups of automorphisms. 
If $p = 3$, the reason is that in their 
equations the coefficient $a_2$ is equal to zero and hence the $j$-invariant of the general member of the 
elliptic pencil vanishes. This implies that the group $H_0$ becomes a group of order $3$ and the full group of 
automorphisms is larger.

We go through the list of automorphism groups of del Pezzo surfaces of degree $1$ in characteristic $0$ given in \cite[Table 8.14]{CAG}
and find those of them which are tame in some characteristic $p > 0$. 
\begin{itemize}
    \item $3 \times (\SL_2(3) : 2)$: This group is tame if $p \neq 3$. 
    By Corollary \ref{cor: possiblegroups}, it is maximal among tame automorphism groups of del Pezzo 
    surfaces of degree $1$. For a suitable choice of coordinates, the unique surface with an action of this group is given by the equation
    $$
    y^2 + x^3 + t_0t_1(t_0^4 - t_1^4) = 0.
    $$
    If $p = 5$, this equation defines the surface of Theorem \ref{p=5} (ii), hence its 
    automorphism group is larger and thus the tame group $3 \times (\SL_2(3) : 2)$ does not occur.
    \item $3 \times (2.D_{12})$: This group is tame if $p \neq 3$ and it is maximal among tame automorphism groups of del Pezzo surfaces of degree $1$. The unique surface with an action of this group is given by the equation
    $$
    y^2 + x^3 + t_0^6 + t_1^6 = 0.
    $$   
    If $p = 5$, this surface is projectively equivalent to the one of the previous case, so, again, 
    its automorphism group is larger. Hence, the group $3 \times (2.D_{12})$ does not occur if $p = 5$.
    \item $6 \times D_6$: This group is tame if $p \neq 3$. 
    The equation given in \cite[Table 8.14]{CAG} defines a one-dimensional 
    family of smooth del Pezzo surfaces with an action of this group if $p\ne 3$. Since the 
    del Pezzo surfaces with a larger automorphism group 
    are unique by \cite[Table 8.14]{CAG} and Theorem \ref{p=5}, a generic member of this family 
   has automorphism group isomorphic to $6 \times D_6$.
    \item $30$: This group is tame if $p \neq 3,5$ and it is maximal among all possible automorphism groups. 
    The equation given in \cite[Table 8.14]{CAG} defines a surface with an action of this group in 
    characteristic $p \neq 3,5$.
    \item $\SL_2(3),2.D_{12},$ and $2 \times 12$: These groups are tame if $p \neq 3$ and the equations given in \cite[Table 8.14]{CAG} with general parameters define surfaces with this automorphism group if $p \neq 3$.
    \item $3 \times D_8$: This group is tame if $p \neq 3$. It appears in the list \cite[Page 519,520]{DI}, but it is missing in \cite[Table 8]{DI} and \cite[Table 8.14]{CAG}. It is realized by the surfaces with Weierstrass equation
    $$
    y^2 + x^3 + t_0t_1(t_0^4 + at_0^2t_1^2 + t_1^4)
    $$
    for a general choice of $a$. This group occurs in every characteristic $p \neq 3$. We denote it by M in Table \ref{tbl:autodp1}.
    \item $20$: This group is tame if $p \neq 5$. If $p \neq 3,5$, the equation given in 
    \cite[Table 8.14]{CAG} defines a surface with this automorphism group, since it 
    is maximal among tame automorphism groups. 

    If $p = 3$, this group does not occur. Indeed, its generator must act by the formula 
    $(t_0,t_1,x,y) \mapsto (t_0,\zeta_{10}t_1,-x,\zeta_4 y)$ as in the case of characteristic $0$, where $\zeta_n$ is a primitive $n$-th root of unity. 
    This easily implies that its equation $y^2+x^3+a_2x^2+a_4x+a_6 = 0$ in characteristic $3$ has the 
    coefficient $a_2 = 0$, hence the $j$-invariant $j(F_\eta)$ of the generic fiber of the elliptic fibration 
    vanishes. But this increases the size of the subgroup $H$ and the automorphism group becomes larger. It follows from 
    Theorem \ref{classification} that the surface is from Case (3) (c) and the automorphism group is $C_6.C_{10}$. 
    \item $D_{16}$: This group is tame. If $p \neq 3$, the equation given in \cite[Table 8.14]{CAG} defines a $1$-dimensional family of surfaces on which this group acts. Since all surfaces whose automorphism group strictly contains $D_{16}$ are unique, the generic member $X$ of this family has $\Aut(X) \cong D_{16}$.

    The same argument as in the previous case shows that the equation of the surface $X$ in characteristic $3$ 
    whose automorphism 
    group contains $D_{16}$ has coefficient $a_2$ equal to zero. This increases the size of the subgroup $H$ of 
    $\Aut(X)$. The surface is from Case 2 (c) in Theorem \ref{classification}. Its group of automorphism group is $C_6.D_8$.
     \item $D_{12}$ and $2 \times 6$: These groups are tame if $p \neq 3$ and the equations given in \cite[Table 8.14]{CAG} with general parameters define surfaces with this automorphism group if $p \neq 3$.

    \item $10$: This group is tame if $p \neq 5$. If $p \neq 3,5$, the equation given in \cite[Table 8.14]{CAG} defines a $1$-dimensional family of surfaces on which this group acts. Since the surfaces whose automorphism groups strictly contain $C_{10}$ are unique, the generic member $X$ of this family realizes the group $C_{10}$.

    If $p = 3$, the equation
    $$
    y^2 + x^3 + t_0^2x^2 + a t_0(t_1^5 + t_0^5) = 0
    $$
    defines a one-dimensional family of surfaces with an action of $C_{10}$ given by $(t_0,t_1,x,y) \mapsto (t_0,\zeta_5 t_1,x,-y)$.

    \item $Q_8, 2 \times 4,$ and $D_8$: These groups are tame and the equations given in \cite[Table 8.14]{CAG} with general parameters define surfaces with these automorphism groups.
    \item $6$: This group is tame if $p \neq 3$ and the equations given in \cite[Table 8.14]{CAG} with general parameters define surfaces with this automorphism group.
    \item $4$ and $2^2$: These groups are tame. The equations given in \cite[Table 8.14]{CAG} define two at least $4$-dimensional families of surfaces with an action of this group. Since the surfaces with larger automorphism group depend on less than three parameters, even in the wild case, a general member of these families will have the desired automorphism group.
    \item $2$: This is the general case that occurs in every characteristic.
\end{itemize}

\subsection{Wild automorphism groups}
Here, we summarize the classification of wild groups of automorphisms of del Pezzo surfaces of degree $1$ by collecting the results of the previous sections. We use the notation of \cite[Table 8.14]{CAG} for wild groups in characteristic $p$ that also occur in characteristic $0$. Let us explain our notation for the remaining cases:

The surfaces in Equation \eqref{eqndp1D10} are reductions modulo $5$ of the surfaces of Type XIII in \cite[Table 8.14]{CAG} after a suitable change of coordinates, hence we call them Type XIII. Similarly, the surface given by Equation \eqref{eqndp1660} is the reduction modulo $5$ of the surfaces of Type I, II, and IV in \cite[Table 8.14]{CAG}.

As for characteristic $p = 3$, consider first the surfaces given by Equation \eqref{eq:3.1-1b}. Recall that they have automorphism group $2 \times 3^2$. Since all the wild automorphisms of order $3$ in this group fix a cuspidal curve on $X$, they are either of conjugacy class $4A_2$ or they arise from a del Pezzo surface $Y$ of degree $2$ given by Equation \eqref{eqndp2C3} and then they are of conjugacy class $3A_2$. Now, comparing eigenvalues of the commuting automorphisms of order $3$ on $\sfE_8$, one checks that the existence of an automorphism of class $4A_2$ would force some other automorphism to be of type $A_2$ or $2A_2$. Hence, all elements in $C_3^2$ are of conjugacy class $3A_2$. The surfaces in Equation \eqref{eq:3.1-1} together with their wild automorphisms are generalizations of Equation \eqref{eq:3.1-1b} and the surfaces in Equation \eqref{eq:3.2-1} are specializations of Equation \eqref{eq:3.1-1}, so all the wild automorphisms in these families are of class $3A_2$.
We call the surfaces given by Equations \eqref{eq:3.1-1} and \eqref{eq:3.2-1} with $a \neq b$ Type XVIII, as in \cite[Table 8,14]{CAG}. Equation \eqref{eq:3.1-1b} has no analogue in characteristic $0$, but it is a specialization of Type XVIII, so we call it XVIII'.
A general equation of Type V in \cite[Table 8.14]{CAG} is smooth in characteristic $3$, hence it defines a del Pezzo surface with an action of $\SL_2(3)$ in characteristic $3$. Therefore, we call the surfaces given by Equation \eqref{eq:3.2-1} with $a = b$ Type V as well. 

The wild automorphisms of order $3$ of the surfaces defined by Equations \eqref{eq:3.2-2} fix the base of the elliptic fibration and act with a unique fixed point on a general fiber, hence they do not leave invariant any $(-1)$-curve on $X$. Therefore, they must be of conjugacy class $4A_2$. Hence, the general surface defined by these equations is the characteristic $3$ analogue of Type XVII of \cite[Table 8.14]{CAG}. The subfamily with $a^2c^4 - b^4 = 0$ has automorphism group $2 \times 6$ with $C_3$ acting trivially on the base of the elliptic fibration, so we call them Type XI.
The surfaces given by Equation \eqref{eq:3.2-2} with $b = c = 0$ are reductions modulo $3$ of the surfaces of Type IX in \cite[Table 8.14]{CAG} and, after a suitable transformation, also of those of Type M.

The general surface given by Equation \eqref{eq:3.2-3} admits a wild automorphism of conjugacy class $4A_2$, by the same argument as in the previous paragraph, so we denote these surfaces by XVII as well. The subfamily with $b = 0$ and $a \neq 0$ has no analogue in characteristic $0$, but they lift together with the action of $C_4$, so we denote them by XIX'. The surface given by Equation \eqref{eq:3.2-3} with $a = b = 0$ lifts to characteristic $0$ together with the action of $C_{20}$, hence it is a reduction modulo $3$ of the surface of Type VIII in \cite[Table 8.14]{CAG}. After a suitable change of coordinates, Type IV \cite[Table 8.14]{CAG} reduces to this surface as well, hence we name it Type IV/VIII

\begin{table}[h!]
    \centering
    \begin{tabular}{|c|c|c|c|c|c|c|c|} \hline
     $p$ &Type&$\Aut(Q)$  & Order &Equation& \# Parameters & Conditions\\ \hline \hline 
     $5$ &XIII&$2.D_{10}$& $20$ &\eqref{eqndp1D10}&$1$&$c\ne 0$\\ \hline
     $5$ &I/II/IV &$6.\PGL_2(5)$ & $720$ & \eqref{eqndp1660} & $0$ & \\ \hline \hline
$3$&XVIII&$C_6$ & $6$ &\eqref{eq:3.1-1}& $3$ & \\ \hline
$3$&XVIII&$C_6$ & $6$ &\eqref{eq:3.2-1}& $2$ & $a \neq b$ \\ \hline
$3$&XVIII'&$2\times 3^2$ & $18$ &\eqref{eq:3.1-1b}& $2$ &  \\ \hline
$3$&XVII&$C_6$ & $6$ &\eqref{eq:3.2-2}& $3$ & $b,c,a^2c^4-b^4\ne 0$ \\ \hline
$3$&XVII&$C_6$& $6$ &\eqref{eq:3.2-3}& $2$ & $b \neq 0$ \\ \hline
$3$&XI&$C_2\times C_6$& $12$ &\eqref{eq:3.2-2}& $2$ & $b,c \ne 0, a^2c^4-b^4 = 0$ \\ \hline
$3$&XIX'&$2.D_6$& $12$ &\eqref{eq:3.2-3}& $1$ & $b = 0, a \neq 0$ \\ \hline
$3$&V&$\SL_2(3)$ & $24$ &\eqref{eq:3.2-1}& $1$ & $a = b = c$ \\ \hline
$3$&IX/M&$C_6.D_8$& $48$ &\eqref{eq:3.2-2}& $1$ & $b=c= 0$ \\ \hline
$3$&IV/VIII&$C_6.C_{10}$& $60$ &\eqref{eq:3.2-3}& $0$ & $a = b = 0$ \\ \hline
         \end{tabular}
    \caption{Wild automorphism groups of del Pezzo surfaces of degree 1}
\end{table}

 \subsection{Conjugacy classes}
As in the case of degree $2$, the conjugacy classes of a tame automorphism $g$ of a del Pezzo surface can be determined using the Lefschetz fixed point formula. The conjugacy classes of wild automorphisms can be determined using the description in the previous section of the corresponding del Pezzo surfaces as reductions modulo $p$ of certain del Pezzo surfaces in characteristic $0$, where the automorphism is tame.
For the surfaces of Type XIX' and XVIII' that have no analogue in characteristic $0$, we determined the conjugacy classes of the wild automorphisms of order $3$ in the previous section. The remaining automorphisms of these surface are obtained by composing with the Bertini involution, and the conjugacy classes of these compositions can be easily determined using Table \ref{tbl:carter}.

We give the classification of all conjugacy classes that can occur in Table \ref{tbl:autodp1}. In that table, we give the name of the relevant family of del Pezzo surfaces in the first column, following \cite[Table 8.14]{CAG}. In the second column, we note the characteristics in which the family occurs. The third and fourth columns give the group $\Aut(X)$ and its order. The remaining columns give the number of elements of a given Carter conjugacy class in $\Aut(X)$.

\begin{remark}
As a concluding remark, we note that, while there are no conjugacy classes in the Weyl groups $W(\sfE_N)$ that are realized in characteristic $p$ but not in characteristic $0$, the tables in the Appendix show that there are certain conjugacy classes that occur in characteristic $0$ but do not occur in characteristic $p$ if $p \in \{3,5,7\}$.

If $p = 7$, the conjugacy classes $A_6$ and $E_7(a_1)$ do not occur, because the Klein quartic becomes singular in characteristic $7$.

If $p = 5$, the conjugacy classes $A_4$ and $E_8(a_2)$ do not occur. In characteristic $0$, the Clebsch cubic surface is the unique del Pezzo surface realizing the conjugacy class $A_4$, and this surface becomes singular in characteristic $5$. The conjugacy class of type $E_8(a_2)$ can only occur on del Pezzo surfaces of degree $1$, where it acts as an automorphism of order $10$ on the base of the elliptic pencil. Since $\PGL_2(\Bbbk)$ contains no element of order $10$ if $p = 5$, this explains why $E_8(a_2)$ does not occur.

If $p = 3$, the conjugacy classes $A_2,A_2 + A_1,A_2 + 2A_1,A_5 + A_2 + A_1,D_4,D_4 + A_2,2D_4,$ $D_5(a_1),$ $E_6 + A_1,E_6(a_1),E_7, $ and $E_8(a_1)$ do not occur in any degree and there are a couple of conjugacy classes that are not realized in all the degrees that they are realized in in characteristic $0$. This is a consequence of several phenomena: Firstly, as above, certain special del Pezzo surfaces in characteristic $0$ become singular in characteristic $3$. Secondly, the group $\PGL_2(\Bbbk)$ contains no elements of order $np$ with $(p,n) = 1$ and $n \geq 2$. Thirdly, there are no wild automorphisms of order $p^2$ on del Pezzo surfaces in any degree by \cite[Theorem 8]{Dolgachev1}.
And finally, the fixed locus of an automorphism $g$ of order $3$ on a del Pezzo surface $X$ in characteristic $3$ tends to be smaller and in a more special position than in characteristic $0$, making it harder to produce blow-ups of $X$ to which $g$ lifts and which are still del Pezzo surfaces.

\end{remark}

 \newpage
\section*{Appendix}
\label{sec:appendix}

\begin{table}[h]
\begin{center}
\begin{tabular}{| l |r| r | r }
\hline
Graph&Order &Characteristic polynomial\\ \hline \hline
$A_k$&$k+1$&$t^k+t^{k-1}+\cdots+1$\\ \hline
$D_k$&$2k-2$&$(t^{k-1}+1)(t+1)$\\ \hline
$D_k(a_1)$&l.c.m$(2k-4,4)$&$(t^{k-2}+1)(t^2+1)$\\ \hline
$D_k(a_2)$&l.c.m$(2k-6,6)$&$(t^{k-3}+1)(t^3+1)$\\ \hline
\vdots&\vdots&\vdots\\ \hline
$D_k(a_{\frac{k}{2}-1})$&even\  $k$ &$(t^{\frac{k}{2}}+1)^2$\\ \hline
$E_6$&12&$(t^4-t^2+1)(t^2+t+1)$\\ \hline
$E_6(a_1)$&9&$t^6+t^3+1$\\ \hline
$E_6(a_2)$&6&$(t^2-t+1)^2(t^2+t+1)$\\ \hline
$E_7$&18&$(t^6-t^3+1)(t+1)$\\ \hline
$E_7(a_1)$&14&$t^7+1$\\ \hline
$E_7(a_2)$&12&$(t^4-t^2+1)(t^3+1)$\\ \hline
$E_7(a_3)$&30&$(t^5+1)(t^2-t+1)$\\ \hline
$E_7(a_4)$&6&$(t^2-t+1)^2(t^3+1)$\\ \hline
$E_8$&30&$t^8+t^7-t^5-t^4-t^3+t+1$\\ \hline
$E_8(a_1)$&24&$t^8-t^4+1$\\ \hline
$E_8(a_2)$&20&$t^8-t^6+t^4-t^2+1$\\ \hline
$E_8(a_3)$&12&$(t^4-t^2+1)^2$\\ \hline
$E_8(a_4)$&18&$(t^6-t^3+1)(t^2-t+1)$\\ \hline
$E_8(a_5)$&15&$t^8-t^7+t^5-t^4+t^3-t+1$\\ \hline
$E_8(a_6)$&10&$(t^4-t^3+t^2-t+1)^2$\\ \hline
$E_8(a_7)$&12&$(t^4-t^2+1)(t^2-t+1)^2$\\ \hline
$E_8(a_8)$&6&$(t^2-t+1)^4$\\ \hline
\end{tabular}
\end{center}
\caption{Carter graphs and characteristic polynomials.}\label{tab1}
\label{tbl:carter}
\end{table}

\vskip5pt
\begingroup
\renewcommand*{\arraystretch}{1.3}
\begin{table}[ht]
\centering
\scalebox{0.7}{
$
\begin{array}{|c|c|c|c|}\hline
\textrm{Order}&\textrm{Jordan form of the induced automorphism of $\bbP^2$}&\textrm{Conjugacy class}&\textrm{Trace}\\ \hline \hline
1/2 & 

\begin{pmatrix}
1 & 0 & 0 \\
0 & 1 & 0 \\
0 & 0 & 1
\end{pmatrix}
&
{\rm id} / 7A_1 & \pm 7 \\
 \hline 
 2/2 &
 
\begin{pmatrix}
1 & 0 & 0 \\
0 & 1 & 0 \\
0 & 0 & -1
\end{pmatrix}

&
4A_1 / 3A_1 & \pm 1 \\
 \hline 
  3/6 &
\begin{pmatrix}
1 & 0 & 0 \\
0 & 1 & 0 \\
0 & 0 & \zeta_3
\end{pmatrix}
\text{ if $p \neq 3$ and }
\begin{pmatrix}
1 & 1 & 0 \\
0 & 1 & 1 \\
0 & 0 & 1
\end{pmatrix}
\text{ if $p = 3$}
&
2A_2 / D_6(a_2) + A_1 & \pm 1 \\
 \hline 
  3/6 &
\begin{pmatrix}
1 & 0 & 0 \\
0 & \zeta_3 & 0 \\
0 & 0 & \zeta_3^2
\end{pmatrix}
\text{ if $p \neq 3$ and }
\begin{pmatrix}
1 & 1 & 0 \\
0 & 1 & 0 \\
0 & 0 & 1
\end{pmatrix}
\text{ if $p = 3$}
&
3A_2 / E_7(a_4) & \pm 2 \\
 \hline 
   4/4 &
\begin{pmatrix}
1 & 0 & 0 \\
0 & \zeta_4 & 0 \\
0 & 0 & -1
\end{pmatrix}
&
2A_3/D_4(a_1) + A_1 & \pm 1 \\  \hline 
   4/4 &
\begin{pmatrix}
1 & 0 & 0 \\
0 & 1 & 0 \\
0 & 0 & \zeta_4
\end{pmatrix}
&
2A_3 + A_1/D_4(a_1)& \pm 3 \\
 \hline 
    6/6 &
\begin{pmatrix}
1 & 0 & 0 \\
0 & \zeta_3 & 0 \\
0 & 0 & -1
\end{pmatrix}
\text{ if $p \neq 3$ and }
\begin{pmatrix}
-1 & 1 & 0 \\
0 & -1 & 1 \\
0 & 0 & -1
\end{pmatrix}
\text{ if $p = 3$}
&
A_5 + A_2/E_6(a_2)& \pm 2 \\
 \hline 
     7/14 &
\begin{pmatrix}
1 & 0 & 0 \\
0 & \zeta_7 & 0 \\
0 & 0 & \zeta_7^4
\end{pmatrix} 
&
A_6/E_7(a_1)& 0 \\
 \hline 
     8/8 &
\begin{pmatrix}
1 & 0 & 0 \\
0 & \zeta_8^3 & 0 \\
0 & 0 & \zeta_8^5
\end{pmatrix} 
&
D_5/D_5+A_1& \pm 1 \\
 \hline 
      9/18 &
\begin{pmatrix}
1 & 0 & 0 \\
0 & \zeta_9^2 & 0 \\
0 & 0 & \zeta_9^3
\end{pmatrix}
&
E_6(a_1)/E_7& \pm 1 \\
 \hline 
      12/12 &
\begin{pmatrix}
1 & 0 & 0 \\
0 & \zeta_{12}^3 & 0 \\
0 & 0 & \zeta_{12}^4
\end{pmatrix}
\text{ if $p \neq 3$ and }
\begin{pmatrix}
\zeta_4 & 1 & 0 \\
0 & \zeta_4 & 1 \\
0 & 0 & \zeta_4
\end{pmatrix}
\text{ if $p = 3$}
&
E_6/E_7(a_2)& 0 \\
 \hline 
\end{array} $
} 
  \caption{Conjugacy classes of elements of $W(\sfE_7)$}\label{table: carter1}
\end{table}
\endgroup

\vskip5pt
\begin{landscape}
\begin{table}[htbp]
\centering
\renewcommand{\arraystretch}{1.3}
\scalebox{0.8}{%
\begin{tabular}{|c|c|cc|ccccccccccc|}
\hline
Name & $\operatorname{char}$ & $\Aut(X)$ & Order & ${\rm id}$ &  $2A_1$ & $4A_1$ & $A_2$& $A_2 + 2A_1$ &$A_3$  & $A_3 + A_1$ & $A_4$ & $D_4$ & $D_4(a_1)$ &$D_5$ \\
\hline \hline
$(\phi,\phi)$ & $\neq 5$ & $2^4 : D_{10}$ & $160$ &  $1$   & $30$ & $5$ & &&  & $40$ & $64$ & & $20$&  \\
& $5$ & $2^4 : (5 : 4)$ & $320$ &  $1$   & $30$ & $5$ & && $80$ & $40$ & $64$ & & $20$& $80$\\ \hline
$(\zeta_3,\zeta_3)$ & $\neq 3$ & $2^4 : \frakS_3$ & $96$ & $1$   & $22$ & $5$ & $8$  &$8$ & & $24$  & & $16$ & $12$& \\ \hline 
$(i,i)$ & $\neq 5$ & $2^4 : 4$ & $64$ & $1$   & $14$ & $5$  & && $16$ & $8$ & & & $4$& $16$ \\
  &$ 5$ & \multicolumn{2}{c|}{\cellcolor{gray!20}(same as $(\phi,\phi)$)} & \multicolumn{11}{c|}{\cellcolor{gray!20}} \\ \hline
$(a,a)$ & any & $2^4 : 2$ & $32$ &  $1$   & $14$ & $5$ & & & & $8$  & & & $4$&  \\ \hline
general & any & $2^4$ & $16$ & $1$  & $10$   & $5$ &  & & & & & & & \\ \hline 
\end{tabular}
}
\caption{Automorphism groups of quartic del Pezzo surfaces}
\label{tbl:autodp4}
\end{table}

\begin{table}[htbp]
\centering
\renewcommand{\arraystretch}{1.3}
\scalebox{0.8}{%
\begin{tabular}{|c|c|cc|cccccccccccccccc|}
\hline
Name & $\operatorname{char}$ & $\Aut(X)$ & Order & ${\rm id}$  & $2A_1$ & $4A_1$ & $A_2$
& $A_2 + 2A_1$ &$2A_2$ & $3A_2$ & $A_3 + A_1$ & $A_4$ & $A_5 + A_1$ & $D_4$ & $D_4(a_1)$  & $D_5$ & $E_6 $ & $E_6(a_1)$ & $E_6(a_2)$ \\ \hline
\hline
I / 3C & $\ne 3$ & $3^3 : \frakS_4$ & $648$ & $1$ & $27$ &$18$  & $6 $ & $54 $& $84$ & $8$ & $162 $&  & $36$ & $36 $&  &  &   &$144 $ & $72$ \\
\hline
II / 5A & $\ne 5$ & $\frakS_5$ & $120$ & $1$ & $15$ & $10 $& & & $20$ &  & $30$ & $24$ &$20$ &  &  &  &    &  & \\ \hline
III / 12A & $\ne 3$ & $\calH_3(3) : 4$ & $108 $& $1$ &  & $9$ &  & & $24$ & $2$ &  &  &  &  & $18$ &  &  $36$  &  & $18$ \\
 & $3$ & $\calH_3(3) : 8$ & $216 $& $1$ &  & $9$ &  && $24$ &$ 2$ &  &   &  &  & $18$ &$ 108$ &   $36$&  & $18$ \\ \hline
IV /  3A & any & $\calH_3(3) : 2$ & $54$ & $1$ &  & $9$  &   & & $24$ & $2$&  &  &  &  & &  &     &  & $18$ \\
\hline
V / 4B & any & $\frakS_4$ & $24$ & $1$ & $3$ &$ 6$ &  & & $8$ &  & $6$ &  &  &  &  &  &    &  & \\ \hline
VI / 6E & any & $\frakS_3 \times \frakS_2$ & $12$ & $1$ & $3$ & $4$ &  & & $2$ &  &  &  & $2$ &  &  &  &    &  & \\ \hline
VII / 8A & $\ne 3$ & $8$ & $8$ & $1$ &  & $1$ &  &  & & &  &   &  &  & $2$ & $4$ &  &  & \\
  &$ 3$ & \multicolumn{2}{c|}{\cellcolor{gray!20}(same as 12A)} & \multicolumn{16}{c|}{\cellcolor{gray!20}} \\ \hline
 VIII / 3D & any & $\frakS_3$ & $6$ & $1$ &  & $3$ &  & & $2$ &  &  &  &  &  &  &  &     &  & \\
\hline
IX / 4A & any & $4$ & $4$ & $1$ &  & $1$ &  &  & & &  &  &  &  & $2$ &  &     &  & \\
\hline
X / 2B & any & $2^2$ & $4$ &$ 1$ & $1 $& $2 $&  &  &  &  &  &  &  &  &  &  &  &  & \\
\hline
XI / 2A & any & $2$ &$ 2 $& $1$ &  & $1$ &  &  &  &  &  &  &  &  &  &  &  &  & \\
\hline
XII / 1A & any & $1$ &$ 1 $&$ 1$ &  &  &  &  &  &  &  &  &  &  &  &  &  &  & \\
\hline
\end{tabular}
}
\caption{Automorphism groups of cubic del Pezzo surfaces}
\label{tbl:cubics}
\end{table}
\end{landscape}

\newgeometry{left=3cm,right=3.4cm,bottom=0.5cm,top=1.6cm}

\begin{landscape}
\begin{table}[htbp]
\centering
\renewcommand{\arraystretch}{1.3}
\scalebox{0.68}{%
\begin{tabular}{|c|c|cc|cccccccccccccccccccccc|}
\hline
Name & $\operatorname{char}$ & $\Aut(X)$ & Order & ${\rm id}$  & $3A_1$ & $4A_1$ & $7A_1$ & $2A_2$ & $3A_2$ &  $2A_3$ & $2A_3 + A_1$ & $A_5 + A_2$ & $A_6$ & $D_4(a_1)$ & $D_4(a_1) + A_1$ & $D_5$ & $D_5 + A_1$ & $D_6(a_2) + A_1$ & $E_6$ &$E_6(a_1)$ &$E_6(a_2)$ & $E_7$ & $E_7(a_1)$ & $E_7(a_2)$ &$E_7(a_4)$ \\
\hline \hline
I & $\ne 3,7$ & $2 \times {\rm L}_2(7)$ & $336$ & $1$ &$ 21 $& $21 $& $1$ & $56 $ & & $42$ &  && $48$ &  &  $42$ & && $56 $& &&  && $48$ &  &  \\ 
& $3$ & $2 \times \PSU_3(9)$ & $12096$ & $1 $& $63 $& $63$ & $1$ & $672$& $56$ & $378$ & $126$ & $504$ & $1728 $& $126$ & $378$ & $1512 $&$ 1512 $& $672$ & $1008$ &  & $504$ & & $1728$ & $1008$ & $56$ \\ \hline 
 II & $\ne 3$ & $2 \times 4^2:\frakS_3$ & $192$ &$ 1$ & $15 $& $15$ & $1$ &$ 32$  & & $18$ &$ 6$ & & & $6$ &  $18 $& $24 $& $24$ &$ 32$ & && & & &  &  \\ 
 & $3$ & \multicolumn{2}{c|}{\cellcolor{gray!20}(same as ${\rm I}$)} & \multicolumn{22}{c|}{\cellcolor{gray!20}} \\ \hline 
 III & $\ne 3$ & $2 \times 4.\mathfrak{A}_4$ &$ 96$ & $1$ & $7$ &$ 7$ & $1$ & &$ 8$& $6$ & $2 $& $8$ & & $2$ & $6$ & && & $16 $& &$ 8$ & & & $16$ & 8 \\
 & $3$ & \multicolumn{2}{c|}{\cellcolor{gray!20}(same as ${\rm I}$)} & \multicolumn{22}{c|}{\cellcolor{gray!20}} \\ \hline 
IV & any & $2 \times \frakS_4$ &$ 48$ &$ 1 $& $9 $& $9 $& $1 $& $8$  & & $6$&  && &  &  $6 $& && $8$ & &&  && &  &  \\ \hline 
V & any & $2 \times 4.2^2$ &$ 32 $& $ 1 $& $7$ & $7$ & $1$ & &  & $6$ & $2 $& & & $2 $& $6$ &&& & & && & & &\\ \hline 
VI & $\ne 3$ & $18$ & $18$ & $1$ & & & $1$ & & $2 $& & & & & & && &&& $6$ & & $6$&  &&  $2$ \\ \hline 
VII & any & $2 \times D_8$ & $16$ & $1$ &$ 5$ &$ 5$ & $1$ & &  & $2$ & & & && $2$ & &&&& & & & & &\\ \hline 
VIII & $\ne 3$ & $2 \times 6$ &$ 12 $& $1$ & $1 $&$ 1 $& $1 $& & $2 $& & & $2$ & && & & && && $ 2$ && &&  $2$ \\
& $3$ & $2 \times (\calH_3(3) : 2)$ & $108$ & $1$ & $9 $& $9 $& $1$ & $24$ & $2$ & & &$ 18$ & & & & & & $24 $& & &  $18$ & & & &  $2 $\\ \hline 
IX & any & $2 \times \frakS_3$ & $12$ & $1$ & $3$ & $3$ & $1 $& $2$ & & & & & &&& & & $2$ && && & &&\\ \hline 
X & any & $2^3$ & $8$&$ 1$ &$ 3$ &$ 3$ &$ 1$ & & & & & & &&&&& & & & && &&\\ \hline 
XI & any & $6$ & $6$ & $1$ & & & $1$ & & $2 $ & &&&  &&& & && & & & & &&$ 2 $\\ \hline 
XII & any & $2^2$ & $4$ & $1 $&$ 1$ &$ 1 $&$ 1$ & & & &&&&& & & && & & & && &\\ \hline 
XIII & any & $2$ &$ 2$ &$ 1 $& & & $1 $& &  &  && & &  && & & & && & && &\\ \hline 
\end{tabular}
}
\caption{Automorphism groups of del Pezzo surfaces of degree $2$}
\label{tbl:autodp2}
\end{table}
\end{landscape}

\begin{landscape}
\begin{table}[htbp]
\centering
\renewcommand{\arraystretch}{1.3}
\scalebox{0.545}{%
\begin{tabular}{|c|c|cc|cccccccccccccccccccccccccc|}
\hline
Name & $\operatorname{char}$ & $\Aut(X)$ & Order & ${\rm id}$  & $4A_1$ & $8A_1$ & $2A_2$  & $3A_2$& $4A_2$ & $2A_3 + A_1$ & $2A_4$ & $A_5 + A_2 + A_1$ & $D_4 + A_2$& $2D_4$ &$D_4(a_1) + A_1$  & $2D_4(a_1)$ & $D_8(a_3)$ & $E_6 + A_1$ & $E_6(a_2)$& $E_6(a_2) + A_2$ & $E_7(a_2)$ &$E_7(a_4) + A_1$ & $E_8$ & $E_8(a_1)$ &$E_8(a_2)$  & $E_8(a_3)$ & $E_8(a_5)$ &$E_8(a_6)$ &$E_8(a_8)$ \\
\hline \hline
I & $\ne 3,5$ & $3 \times (\SL_2(3):2)$& $144$ &$ 1$ &$ 12$ & $1$ & $8$ & $16 $& $2$ &  &  &  &  & $8 $&  & $6$ & $12$& &  & $24$ &  & $16 $ &  & $24$ &  &$ 12  $&  &  &  $2$  \\
 & $5$ & $6.\PGL_2(5)$& $ 720 $&$ 1$ &$ 20 $& $1$ & $20$ & $40 $& $2$ &  &$ 24 $& $40 $& $40 $& $20$ &  & $30$ & $60$ & & $40$ & $40 $ &  &$ 40 $ & $48 $& $120$ &  &$ 60 $ &$ 48$ & $24 $& $ 2$\\ \hline 

II & $\ne 3,5$ & $3 \times 2.D_{12}$& $72$ &$ 1 $& $8$ & $1$ &$ 2$ &$ 4 $& $2$ & & & $4$ &$ 4 $& $2$ &  &$ 6$ & & & $4$ & $16 $& &  $4  $& & & & $12$ & & &  $2$ \\ 
& $5$ & \multicolumn{2}{c|}{\cellcolor{gray!20}(same as ${\rm I}$)} & \multicolumn{26}{c|}{\cellcolor{gray!20}} \\ \hline 

III & $\ne 3$ & $6 \times D_6$& $36 $&$ 1$ & $6 $&$ 1$ & $2$ & $4$ & $2$ & & & & & $2$ &  & & & & & $12 $& & $4$  & & & & & & & $2$ \\ \hline 

IV & $\ne 3,5$ & $30$& $30$ & $1$ &  & $1$& &  & $2$ &   & $4$ & & & &  & & & &  & &  &  & $8$ & & & & $8$ &  $4 $& $2$\\ 
& $5$ & \multicolumn{2}{c|}{\cellcolor{gray!20}(same as ${\rm I}$)} & \multicolumn{26}{c|}{\cellcolor{gray!20}} \\ 
& $3$ & $6.10$& $60$ & $1 $&  & $1$ & & & $2$ &  & $4$& & & & &$6$ &  & & & & & & $8$ & & $24$ & & $8$ & $4$ & $2 $\\ \hline 

M  & $\ne 3$ & $3 \times D_8$& $24$ &$ 1$ &$ 4 $& $1$ & &  & $2$ & & & & & & & $2$ &  & & & $8$ & & & & & & $4$ & & & $2$\\ 
& $3$ & $6.D_8$&$ 48 $& $1$ & $8 $& $1$ & &  &$ 2$ & & & & & & & $2$ &$ 12$ & & & $16$ & & & & & & $4$ & & & $2$\\  \hline 

V & any & $\SL_2(3)$& $24$ & $1 $& & $1 $& &$ 8$ & & & & & & & & $6$ & & & & & & $8$ & & & & & & & \\ \hline 

VI & $\ne 3$ & $2.D_{12}$ & $24$ & $1$ &$ 8$ &$ 1$ & $2$ & & & & &  & $4$ & $2$  &  & $6$ & & & & & &  & & & & & & & \\ \hline 

VII & $\ne 3$ & $2 \times 12$& $24 $& $1$ &$ 2 $& $1$ & &$ 2$ & &$ 2$ &  &$ 2$ & & & $2 $& & &$4$  & $2$ & & $4$ &$ 2$ & & & & & & & \\ \hline 

VIII & $\ne 3,5$ & $20$& $20$ &$ 1 $&  &$ 1$ & &  & &  & $4$& & & & & $2$ &  & & & & & & & & $8$ & & &$ 4 $& \\
 & $3$ & \multicolumn{2}{c|}{\cellcolor{gray!20}(same as ${\rm IV}$)} & \multicolumn{26}{c|}{\cellcolor{gray!20}} \\ \hline 

IX & $\ne 3$ & $D_{16}$& $16 $&$ 1$ & $8 $& $1$ & & &  & & & & & & &$ 2$ & $4$ & & &  & & & & & & & &&  \\ 
& $3$ & \multicolumn{2}{c|}{\cellcolor{gray!20}(same as ${\rm M}$)} & \multicolumn{26}{c|}{\cellcolor{gray!20}} \\ \hline 

X & $\ne 3$ & $D_{12}$&$ 12 $& $1$ &$ 6$ & $1$ & $2$ & & & & & & & $2$ &  & & & & & & &  & & & & & & & \\ \hline 

XI & any & $2 \times 6$& $12 $& $1 $& $2$ & $1 $& & &$ 2$ & & & & & & & & & & & $4$ & & & & & & & & & $2 $\\ \hline 

XII & $\ne 3$ & $2 \times 6$& $12$& $1$ & $2$ & $1$ &  &$ 2$ & & & &$2$ & & & & & & & $2 $& & &$ 2$ & & & & & & & \\ \hline 

XIII & $\neq 5$ & $10$&$ 10$ & $1$ & & $1$ & & & & & $4$& & & & & & & & & & & & & & & & & $4$ & \\
 & $5$ & $2.D_{10}$& $20$ &$ 1$ & &$ 1$ & & & & & $4$& & & & & $10$ & & & & & & & & & & & & $4 $&  \\ \hline 

XIV & any & $Q_8$& $8$ & $1$ & & $1$ & & & & & & & & & & $6$ & & & & & & & & & & & & & \\ \hline 
 
XV & any & $2 \times 4$& $8$ & $1$ & $2$ & $1$ & & & & $2 $&  & & & &$ 2 $& & & & & & & & & &  & & & & \\ \hline

XVI & any & $D_8$&$ 8 $& $1$& $4$ & $1 $& & & & & & & & & & $2 $& & & & & & & & & &  & & & \\ \hline 

XVII & any & $6$& $6$ &$ 1$ &  & $1$ & & & $2$ & & & & & & & & & & & & & & & & & & & & $2$ \\ \hline 

XVIII & any & $6$&$ 6 $& $1 $&  & $1 $& & $2$ & & & & & & & & & & & & & & $2 $& & & & & & & \\ \hline 

XVIII' & $3$ & $2.3^2$& $18$ & $1$ &  & $1$ & & $8$ & & & & & & & & & & & & & & $8$ & & & & & & &  \\ \hline 

XIX & any & $4$& $4 $& $1$& & $1$ & & & & & & & &  & & $2$ & & & & & & & & & & & && \\ \hline 
XIX' & $3$ & $2.D_6$& $12$ &$ 1$ & & $1$& & &$ 2$& & & & & & & $6$ & & & & & & & & & & & & & $2 $\\ \hline 

XX & any & $2^2$ & $4$ & $1$ & $2$ & $1$ & & & & & & & & & & & & & & & & & & & & & & &  \\ \hline 

XXI & any & $2$ &$ 2 $&$ 1 $& &$1 $& & &  & & & & & & & & & & & & & & & & & & & & \\ \hline

\end{tabular}
}
\caption{Automorphism groups of del Pezzo surfaces of degree $1$}
\label{tbl:autodp1}
\end{table}
\end{landscape}

\restoregeometry

\end{document}